\documentclass{gtart}

\usepackage[inner=20mm, outer=20mm, textheight=245mm]{geometry}

\usepackage{psfrag, graphicx, subfigure, epsfig,color}

\usepackage{amsmath, amssymb, latexsym, euscript, amsthm}

\usepackage{mathptmx, wrapfig}

\usepackage{mathrsfs}

\usepackage{enumerate}

\usepackage{booktabs}

\def\bkR{{\rm I\kern-.17em R}}

\def\bkZ{{\rm Z\kern-.26em Z}}

\DeclareMathOperator{\WS}{WS}
\def\PSH{\Sigma^3}

\theoremstyle{plain}
\newtheorem{theorem}{Theorem}
\newtheorem*{theorem*}{Theorem}
\newtheorem{lemma}[theorem]{Lemma}

\newtheorem*{claim*}{Claim}

\theoremstyle{definition}
\newtheorem{definition}[theorem]{Definition}
\newtheorem*{definition*}{Definition}

\theoremstyle{remark}
\newtheorem{remark}[theorem]{Remark}

\numberwithin{equation}{section}

\begin{document}

\title{Unravelling the Dodecahedral Spaces}
\author{Jonathan Spreer and Stephan Tillmann}

\begin{abstract}
The hyperbolic dodecahedral space of Weber and Seifert has a natural non-positively curved cubulation obtained by subdividing the dodecahedron into cubes. We show that the hyperbolic dodecahedral space has a $6$--sheeted irregular cover with the property that the canonical hypersurfaces made up of the mid-cubes give a very short hierarchy. Moreover, we describe a $60$--sheeted cover in which the associated cubulation is special.
We also describe the natural cubulation and covers of the spherical dodecahedral space (aka Poincar\'e homology sphere).
\end{abstract}



\maketitle


\section{Introduction}

A \emph{cubing} of a 3--manifold $M$ is a decomposition of $M$ into Euclidean cubes identified along their faces by Euclidean isometries. This gives $M$ a singular Euclidean metric, with the singular set contained in the union of all edges. The cubing is \emph{non-positively curved} if the dihedral angle along each edge in $M$ is at least $2\pi$ and each vertex satisfies Gromov's \emph{link condition}:
The link of each vertex is a triangulated sphere in which each 1--cycle consists of at least 3 edges, and if a 1--cycle consists of exactly 3 edges, then it bounds a unique triangle. In this case, we say that $M$ has an \emph{NPC cubing}.

The universal cover of an NPC cubed 3--manifold is CAT(0). Aitchison, Matsumoto and Rubinstein~\cite{AMR} showed by a direct construction that if each edge in an NPC cubed 3--manifold has even degree, then the manifold is virtually Haken. Moreover, Aitchison and Rubinstein~\cite{AM-Bull1999} showed that if each edge degree in such a cubing is a multiple of four, then the manifold is virtually fibred. 

A cube contains three canonical squares (or $2$-dimensional cubes), each of which is parallel to two sides of the cube and cuts the cube into equal parts. These are called \emph{mid-cubes}. The collection of all mid-cubes gives an immersed surface in the cubed 3--manifold $M$, called the \emph{canonical (immersed) surface}. If the cubing is NPC, then each connected component of this immersed surface is $\pi_1$--injective. If one could show that one of these surface subgroups is separable in $\pi_1(M),$ then a well-known argument due to Scott~\cite{Scott} shows that there is a finite cover of $M$ containing an embedded $\pi_1$--injective surface, and hence $M$ is virtually Haken. In the case where the cube complex is \emph{special} (see \S\ref{sec:Cube complexes}), a \emph{canonical completion and retraction construction} due to Haglund and Wise~\cite{HW} shows that these surface subgroups are indeed separable because the surfaces are convex. Whence a 3--manifold with a special NPC cubing is virtually Haken.  The missing piece is thus to show that an NPC cubed 3--manifold has a finite cover such that the lifted cubing is special. This is achieved in the case where the fundamental group of the 3--manifold is hyperbolic by the following cornerstone in Agol's proof of Waldhausen's Virtual Haken Conjecture from 1968:

\begin{theorem}[Virtual special; Agol~\cite{Agol}, Thm 1.1] 
Let $G$ be a hyperbolic group which acts properly and cocompactly on a $CAT(0)$ cube complex $X.$ Then $G$ has a finite index subgroup $F$ so that $X/F$ is a special cube complex.
\end{theorem}

In general, it is known through work of Bergeron and Wise~\cite{BW} that if $M$ is a closed hyperbolic 3--manifold, then $\pi_1(M)$  is isomorphic to the fundamental group of an NPC cube complex. However, the dimension of this cube complex may be arbitrarily large and it may not be a manifold. Agol's theorem provides a finite cover that is a special cube complex, and the $\pi_1$--injective surfaces of Kahn and Markovic~\cite{KM} are quasi-convex and hence have separable fundamental group. Thus, the above outline completes a sketch of the proof that $M$ is virtually Haken. An embedding theorem of Haglund and Wise~\cite{HW} and Agol's virtual fibring criterion~\cite{Agol-RFRS} then imply that $M$ is also virtually fibred.

Weber and Seifert~\cite{WS1933} described two closed 3--manifolds that are obtained by taking a regular dodecahedron in a space of constant curvature and identifying opposite sides by isometries. One is hyperbolic and known as the \emph{Weber-Seifert dodecahedral space} and the other is spherical and known as the \emph{Poincar\'e homology sphere}. Moreover, antipodal identification on the boundary of the dodecahedron yields a third closed $3$-manifold which naturally fits into this family: the real projective space. 

The dodecahedron has a natural decomposition into 20 cubes, which is a NPC cubing in the case of the Weber-Seifert dodecahedral space.
The main result of this note can be stated as follows.

\begin{theorem}
  \label{thm:main}
  The hyperbolic dodecahedral space $\WS$ of Weber and Seifert admits a cover of degree $60$ in which the lifted natural cubulation of $\WS$ is special.
\end{theorem}

In addition, we exhibit a $6$--sheeted cover of $\WS$ in which the canonical immersed surface consists of six embedded surface components and thus gives a very short hierarchy of $\widetilde{\WS}.$ The special cover from Theorem~\ref{thm:main} is the smallest regular cover of $\WS$ that is also a cover of this $6$--sheeted cover. Moreover, it is the smallest regular cover of $\WS$ that is also a cover of the $5$-sheeted cover with positive first Betti number described by Hempel~\cite{He}.

We conclude this introduction by giving an outline of this note. The dodecahedral spaces are described in \S\ref{sec:Dodecahedral spaces}. Covers of the hyperbolic dodecahedral space are described in \S\ref{sec:covers of WS}, and all covers of the spherical dodecahedral space and the real projective space in \S\ref{sec:covers of PHS}.

\medskip

\textbf{Acknowledgements:}
Research of the first author was supported by the Einstein Foundation (project “Einstein Visiting Fellow Santos”).
Research of the second author was supported in part under the Australian Research Council's Discovery funding scheme (project number DP160104502). The authors thank Schloss Dagstuhl Leibniz-Zentrum f\"ur Informatik and the organisers of Seminar 17072, where this work was completed. 

The authors thank Daniel Groves and Alan Reid for their encouragement to write up these results, and the anonymous referee for some insightful questions and comments which triggered us to find a special cover.




\section{Cube complexes, injective surfaces and hierarchies}
\label{sec:Cube complexes}

A \emph{cube complex} is a space obtained by gluing Euclidean cubes of edge length one along subcubes. A cube complex is $CAT(0)$ if it is $CAT(0)$ as a metric space, and it is \emph{non-positively curved (NPC)} if its universal cover is $CAT(0).$ Gromov observed that a cube complex is NPC if and only if the link of each vertex is a flag complex.

We identify each $n$--cube as a copy of $[-\frac{1}{2}, \frac{1}{2}]^n$. A mid-cube in $[-\frac{1}{2}, \frac{1}{2}]^n$ is the intersection with a coordinate plane $x_k=0.$ If $X$ is a cube complex, then a new cube complex $Y$ is formed by taking one $(n-1)$--cube for each midcube of $X$ and identifying these $(n-1)$--cubes along faces according to the intersections of faces of the corresponding $n$--cubes. The connected components of $Y$ are the \emph{hyperplanes} of $X,$ and each hyperplane $H$ comes with a canonical immersion $H \to X.$ The image of the immersion is termed an \emph{immersed hyperplane} in $X.$ If $X$ is $CAT(0),$ then each hyperplane is totally geodesic and hence embedded.

The NPC cube complex $X$ is \emph{special} if 
\begin{enumerate}
\item Each immersed hyperplane embeds in $X$ (and hence the term ``immersed" will henceforth be omitted).
\item Each hyperplane is 2--sided.
\item No hyperplane self-osculates.
\item No two hyperplanes inter-osculate.
\end{enumerate}

The prohibited pathologies are shown in Figure~\ref{fig:special} and are explained now. An edge in $X$ is dual to a mid-cube if it intersects the midcube. We say that the edge of $X$ is dual to the hyperplane $H$ if it intersects its image in $X.$ The hyperplane dual to edge $a$ is unique and denoted $H(a).$ Suppose the immersed hyperplane is embedded. It is 2--sided if one can consistently orient all dual edges so that all edges on opposite sides of a square have the same direction. Using this direction on the edges, $H$ \emph{self-osculates} if it is dual to two distinct edges with the same initial or terminal vertex. Hyperplanes $H_1$ and $H_2$ \emph{inter-osculate} if they cross and they have dual edges that share a vertex but do not lie in a common square.

\begin{figure}[htb]
	\begin{center}
 		{\includegraphics[width=0.9\textwidth]{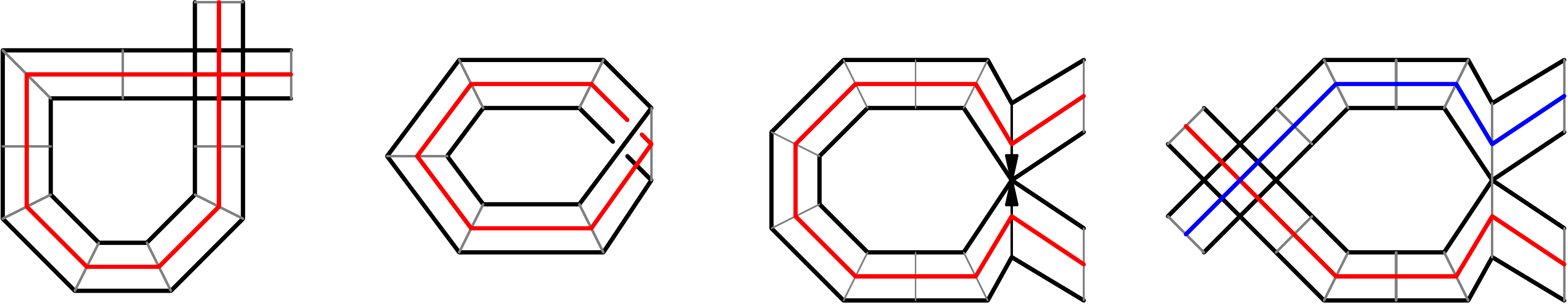}} 
	\end{center}
	\caption{not embedded; 1-sided; self-osculating; inter-osculating (compare to~\cite[Figure 4.2]{Wise})\label{fig:special}}
\end{figure}

The situation is particularly nice in the case where the NPC cube complex $X$ is homeomorphic to a 3--manifold. Work of Aitchison and Rubinstein (see \S3 in \cite{AR1999}) shows that each immersed hyperplane is mapped $\pi_1$--injectively into $X.$ Hence if one hyperplane is embedded and 2--sided, then $X$ is a Haken 3--manifold. 
Moreover, if each hyperplane embeds and is 2--sided, then one obtains a hierarchy for $X.$ This is well-known and implicit in \cite{AR1999}. One may first cut along a maximal union of pairwise disjoint hypersurfaces to obtain a manifold $X_1$ (possibly disconnected) with incompressible boundary. Then each of the remaining hypersurfaces gives a properly embedded surface in $X_1$ that is incompressible and boundary incompressible. This process iterates until one has cut open $X$ along all the mid-cubes, and hence it terminates with a collection of balls. In particular, if $Y$ consists of three pairwise disjoint (not necessarily connected) surfaces, each of which is embedded and 2--sided, then one has a very short hierarchy.


\section{The dodecahedral spaces}
\label{sec:Dodecahedral spaces}

The main topic of this paper is a study of low-degree covers of the hyperbolic dodecahedral space. However, we also take the opportunity to extend this study to the spherical dodecahedral space in the hope that this will be a useful reference. When the sides are viewed combinatorially, there is a third dodecahedral space which naturally fits into this family and again gives a spherical space form: the real projective space. The combinatorics of these spaces is described in this section.


\subsection{The Weber-Seifert Dodecahedral space}

The Weber-Seifert Dodecahedral space $\WS$ is obtained by gluing the opposite faces of a dodecahedron with a $3 \pi / 5$-twist. This yields a decomposition $\mathcal{D}_{\WS}$ of the space into one vertex, six edges, six pentagons, and one cell (see Figure~\ref{fig:dodecahedron} on the left). The dodecahedron can be decomposed into $20$ cubes by a) placing a vertex at the centre of each edge, face, and the dodecahedron, and b) placing each cube around one of the $20$ vertices of the dodecahedron with the other seven vertices in the centres of the three adjacent edges, three adjacent pentagons, and the center of the dodecahedron. Observe that identification of opposite faces of the original dodecahedron with a $3 \pi / 5$-twist yields a $14$-vertex, $54$-edge, $60$ square, $20$-cube decomposition $\hat{\mathcal{D}}_{\WS}$ of $\WS$ (see Figure~\ref{fig:dodecahedron} on the right). Observe that every edge of $\hat{\mathcal{D}}_{\WS}$ occurs in $\geq 4$ cubes, and each vertex satisfies the link condition. We therefore have an NPC cubing.

\begin{figure}[htb]
	\begin{center}
		\raisebox{0.5cm}{\includegraphics[height=7cm]{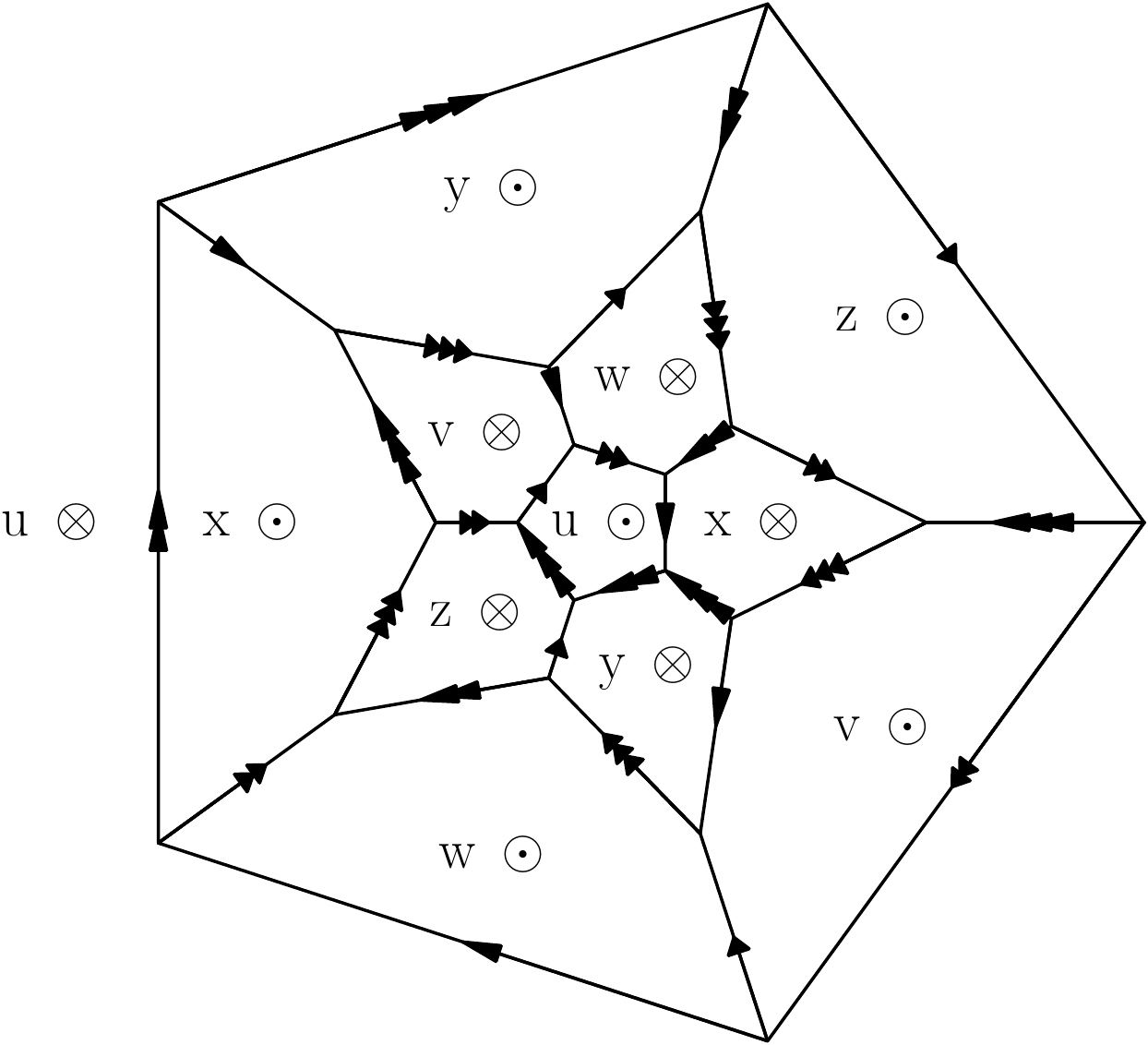}} \hspace{1cm} \includegraphics[height=8cm]{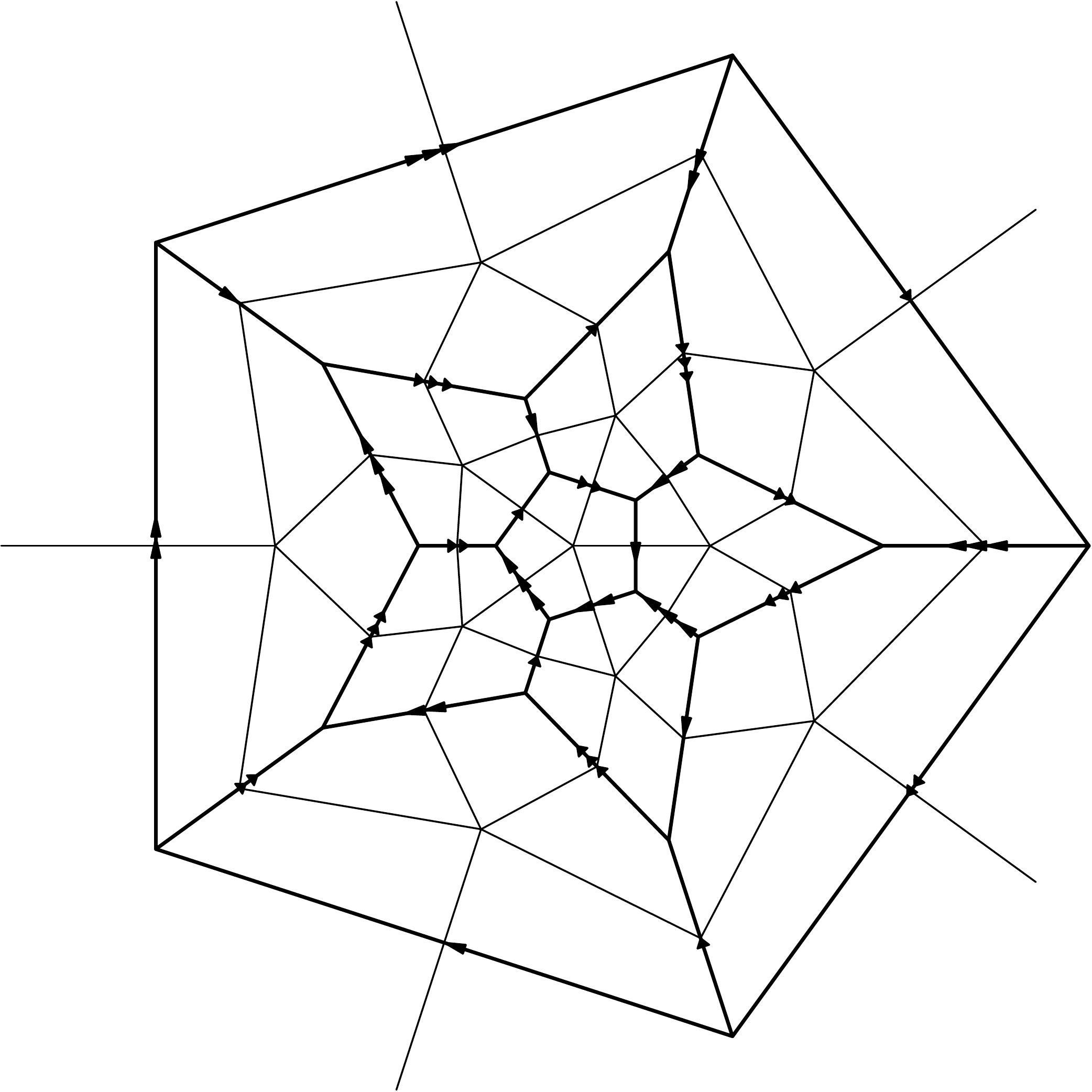}
	\end{center}
	\caption{Left: face and edge-identifications on the dodecahedron yielding the Weber-Seifert dodecahedron space. Right: decomposition of the Weber-Seifert dodecahedral space into $20$ cubes.\label{fig:dodecahedron}}
\end{figure}

\begin{figure}[htb]
	\begin{center}
		\raisebox{2cm}{\includegraphics[height=5cm]{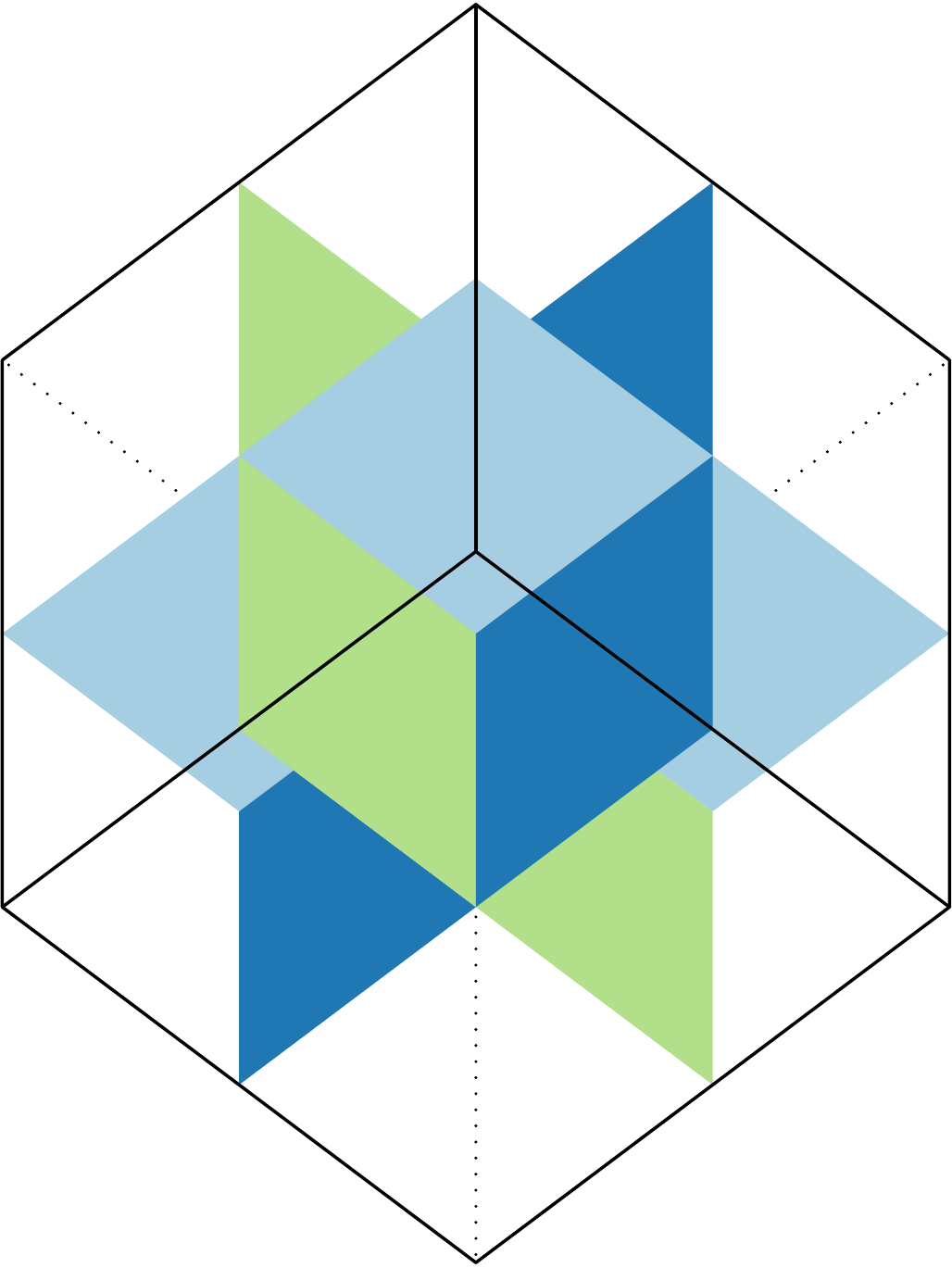}} \hspace{3cm} \includegraphics[height=8cm]{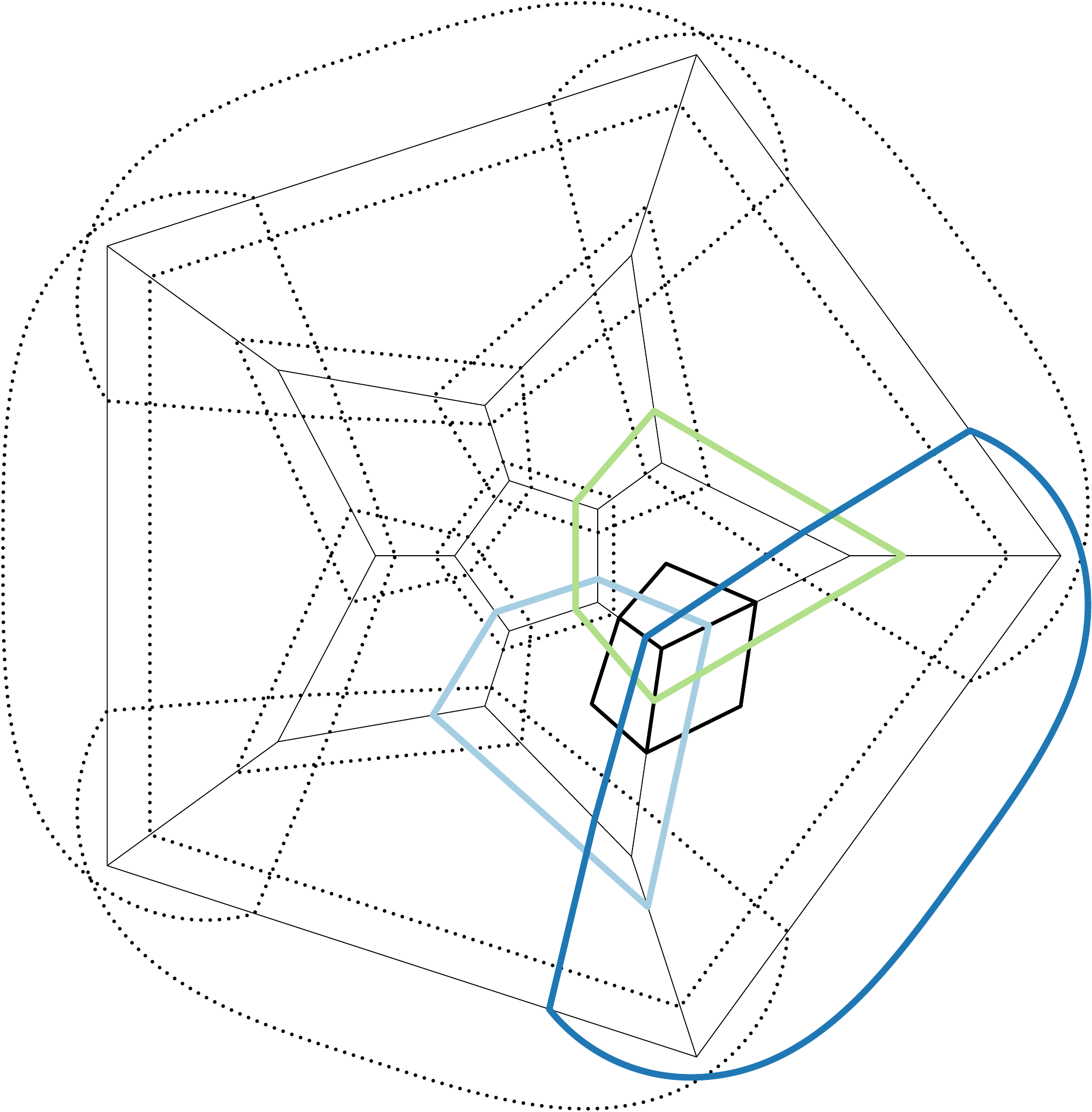}
	\end{center}
	\caption{Left: Immersed canonical surface in one cube. Right: intersection pattern of one cube, the immersed canonical surface, and the boundary of the dodecahedron in $\mathcal{D}_{\WS}$.\label{fig:immersed}}
\end{figure}

The mid-cubes form pentagons parallel to the faces of the dodecahedron, and under the face pairings glue up to give a 2--sided immersed surface of genus four. We wish to construct a cover in which the canonical surface splits into embedded components -- which neither self-osculate with themselves, nor inter-osculate with other surface components.

\subsection{The Poincar\'e homology sphere}

The Poincar\'e homology sphere $\PSH$ is obtained from the dodecahedron by gluing opposite faces by a $ \pi / 5$-twist. This results in a decomposition $\mathcal{D}_{\PSH}$ of $\PSH$ into one vertex, ten edges, six pentagons, and one cell (see Figure~\ref{fig:poincare} on the left). Again, we can decompose $\mathcal{D}_{\PSH}$ into $20$ cubes. Note, however, that in this case some of the cube-edges only have degree three (the ones coming from the edges of the original dodecahedron). This is to be expected since $\PSH$ supports a spherical geometry.

\begin{figure}[htb]
	\begin{center}
		\raisebox{0.5cm}{\includegraphics[height=7cm]{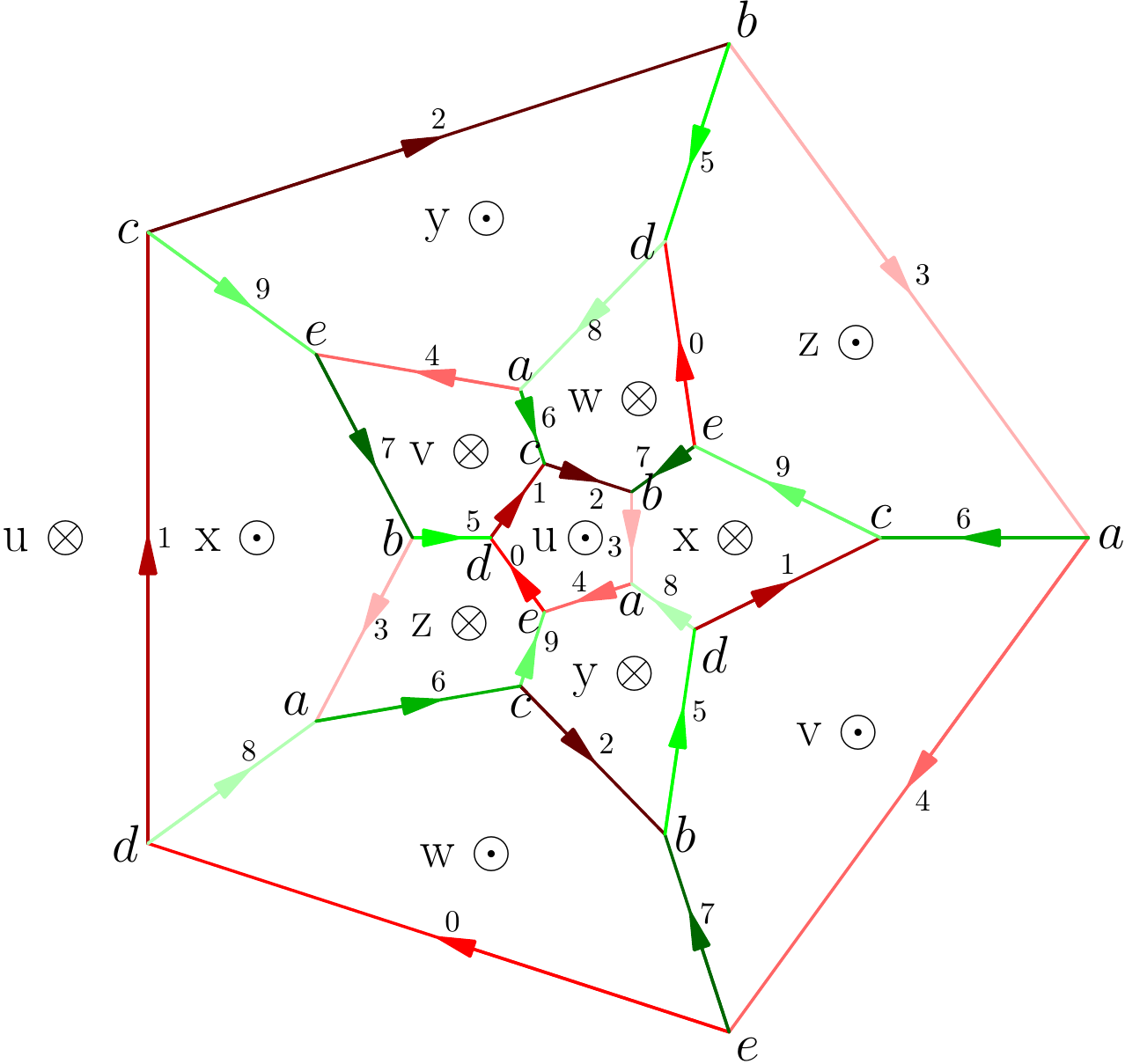}} \hspace{1cm} \includegraphics[height=7cm]{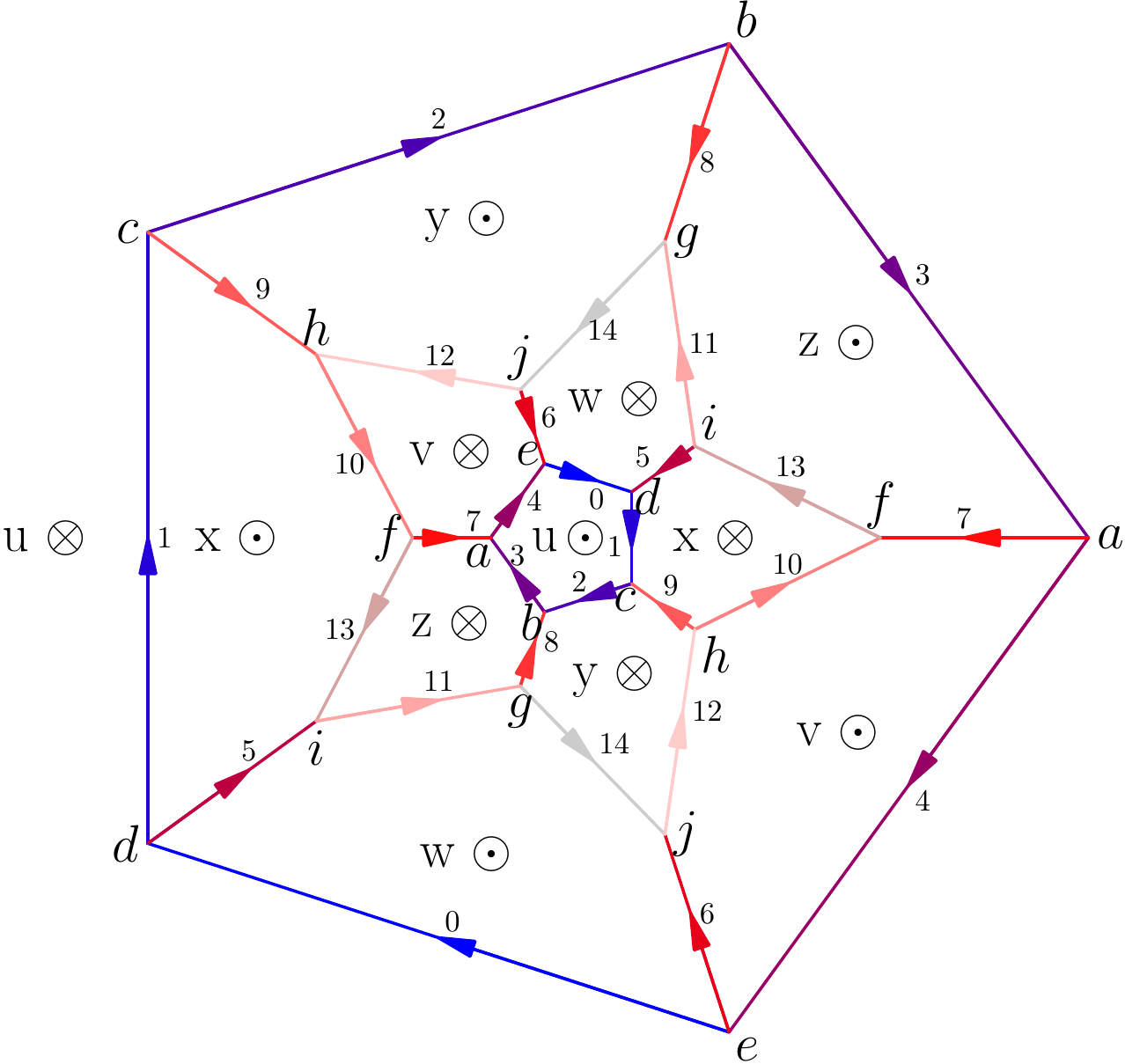}
	\end{center}
	\caption{Left: face and edge-identifications on the dodecahedron yielding the Poincar\'e homology sphere. Right: face and edge-identifications on the dodecahedron yielding the real projective space.\label{fig:poincare}}
\end{figure}

\subsection{Real projective space}

Identifying opposite faces of the dodecahedron by a twist of $ \pi$ results in identifying antipodal points of a $3$-ball (see Figure~\ref{fig:poincare} on the right). Hence, the result is a decomposition $\mathcal{D}_{\mathbb{R}P^3}$of $\mathbb{RP}^3$ into ten vertices, $15$ edges, six faces, and one cell. As in the above cases, this decomposition can be decomposed into $20$ cubes, with some of the cube-edges being of degree two.


\section{Covers of the Weber-Seifert space}
\label{sec:covers of WS}

In order to obtain a complete list of all small covers of the Weber-Seifert space $\WS$, we need a list of all low index subgroups of $\pi_1 (\WS)$ in a presentation compatible with $\mathcal{D}_{\WS}$ and its cube decomposition $\hat{\mathcal{D}}_{\WS}$.

The complex $\mathcal{D}_{\WS}$ has six pentagons $u$, $v$, $w$, $x$, $y$, and $z$. These correspond to antipodal pairs of pentagons in the original dodecahedron, see Figure~\ref{fig:dodecahedron} on the left. Passing to the dual decomposition, these six pentagons corresponds to loops which naturally generate $\pi_1 (\WS)$. The six edges of $\mathcal{D}_{\WS}$
\includegraphics[height=.3cm]{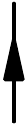},
\includegraphics[height=.3cm]{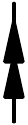},
\includegraphics[height=.3cm]{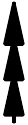},
\includegraphics[height=.3cm]{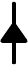},
\includegraphics[height=.3cm]{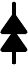}, and
\includegraphics[height=.3cm]{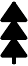}
each give rise to a relator in this presentation of the fundamental group of $\WS$ in the following way: fix edge \includegraphics[height=.3cm]{four} and start at a pentagon containing \includegraphics[height=.3cm]{four}, say $u$. We start at the pentagon labelled $u$ with a back of an arrow $\otimes$ -- the outside in Figure~\ref{fig:dodecahedron} on the left. We traverse the dodecahedron, resurface on the other pentagon labelled $u$ with an arrowhead $\odot$ (the innermost pentagon in Figure~\ref{fig:dodecahedron}). We then continue with the unique pentagon adjacent to the center pentagon along edge \includegraphics[height=.3cm]{four}. In this case $v$ labelled with the tail of an arrow, we traverse the dodecahedron, resurface at $(v,\odot)$, and continue with $(w,\odot)$ which we follow through the dodecahedron in reverse direction, and so on. After five such traversals we end up at the outer face where we started. The relator is now given by the labels of the pentagons we encountered, taking into account their orientation (arrowhead or tail). In this case the relator is $r($\includegraphics[height=.3cm]{four}$)= uvw^{-1}y^{-1}z$.

Altogether we are left with 
$$ \begin{array}{llll} \pi_1 (\WS) = \langle \,\, u,\, v,\, w,\, x,\, y,\, z \, \mid & 
    uxy^{-1}v^{-1}w, & uyz^{-1}w^{-1}x, & uzv^{-1}x^{-1}y, \\
  &uvw^{-1}y^{-1}z, & uwx^{-1}z^{-1}v, & vxzwy \,\, \rangle .
  \end{array} $$

Using this particular representation of the fundamental group of the Weber-Seifert dodecahedral space we compute subgroups of $\pi_1 (\WS)$ of index $k$ ($k < 10$) via \texttt{GAP} function \texttt{LowIndexSubgroupsFpGroup} \cite{GAP}, and \texttt{Magma} function \texttt{LowIndexSubgroups} \cite{Magma} and use their structure to obtain explicit descriptions of their coset actions (using \texttt{GAP} function \texttt{FactorCosetAction} \cite{GAP}) which, in turn, can be transformed into a gluing table of $k$ copies of the dodecahedron (or $20k$ copies of the cube). Given such a particular decomposition, we can track how the canonical surface evolves and whether it splits into embedded components.

We provide a {\em GAP} script for download from \cite{ST}. The script takes a list of subgroups as input (presented each by a list of generators from $\pi_1 (\WS)$) and computes an array of data associated to the corresponding covers of $\mathcal{D}_{\WS}$. The script comes with a sample input file containing all subgroups of $\pi_1 (\WS)$ of index less than ten. The subgroups are presented in a form compatible with the definition of $\pi (\WS)$ discussed above. 


\subsection{Covers of degree up to five}
\label{ssec:five}

A computer search reveals that there are no covers of degrees $2$, $3$, and $4$, and $38$ covers of degree $5$. Their homology groups are listed in Table~\ref{tab:1to9}. For none of them, the canonical surface splits into embedded components. Moreover, in all but one case it does not even split into multiple immersed components, with the exception being the $5$-sheeted cover with positive first Betti number described by Hempel~\cite{He}, where it splits into five immersed components. 


\subsection{Covers of degree six}
\label{ssec:six}

There are $61$ covers of degree six, for $60$ of which the canonical surface does not split into multiple connected components (see Table~\ref{tab:1to9} below for their first homology groups, obtained using \texttt{GAP} function \texttt{AbelianInvariants} \cite{GAP}). However, the single remaining example leads to an irregular cover $\mathcal{C}$ with deck transformation group isomorphic to $\operatorname{A}_5$, for which the canonical surface splits into six embedded components. The cover is thus a Haken cover (although this fact also follows from the first integral homology group of $\mathcal{C}$ which is isomorphic to $\mathbb{Z}^5 \oplus \mathbb{Z}_2^2 \oplus \mathbb{Z}_5^3$, see also Table~\ref{tab:1to9}), and the canonical surface defines a very short hierarchy.

The subgroup is generated by 

$$ \begin{array}{lllll}  
    u, & v^{-1}w^{-1}, & w^{-1}x^{-1}, & x^{-1}y^{-1}, & y^{-1}z^{-1}, \\
    z^{-1}v^{-1}, & vuy^{-1}, & v^2z^{-1}, & vwy^{-1}, & vxv^{-1} 
  \end{array} $$

and the complex is given by gluing six copies $1, 2, \ldots , 6$ of the dodecahedron with the orbits for the six faces as shown in Figure~\ref{fig:fpg} on the left (the orientation of the orbit is given as in Figure~\ref{fig:dodecahedron} on the left). The dual graph of $\mathcal{C}$ (with one vertex for each dodecahedron, one edge for each gluing along a pentagon, and one colour per face class in the base $\mathcal{D}_{\WS}$) is given in Figure~\ref{fig:fpg} on the right.

\begin{figure}[b!]
  \begin{center}
    \begin{tabular}{ll}
	\raisebox{3.5cm}{
	\begin{tabular}{l|l}
        \toprule
	face & orbit \\
	\midrule
	$u$&$(2,5,3,6,4)$ \\
	$v$&$(1,2,6,4,3)$ \\
	$w$&$(1,3,2,5,4)$ \\
	$x$&$(1,4,3,6,5)$ \\
	$y$&$(1,5,4,2,6)$ \\
	$z$&$(1,6,5,3,2)$ \\
        \bottomrule
	\end{tabular} } & \hspace{2cm} \includegraphics[height=7cm]{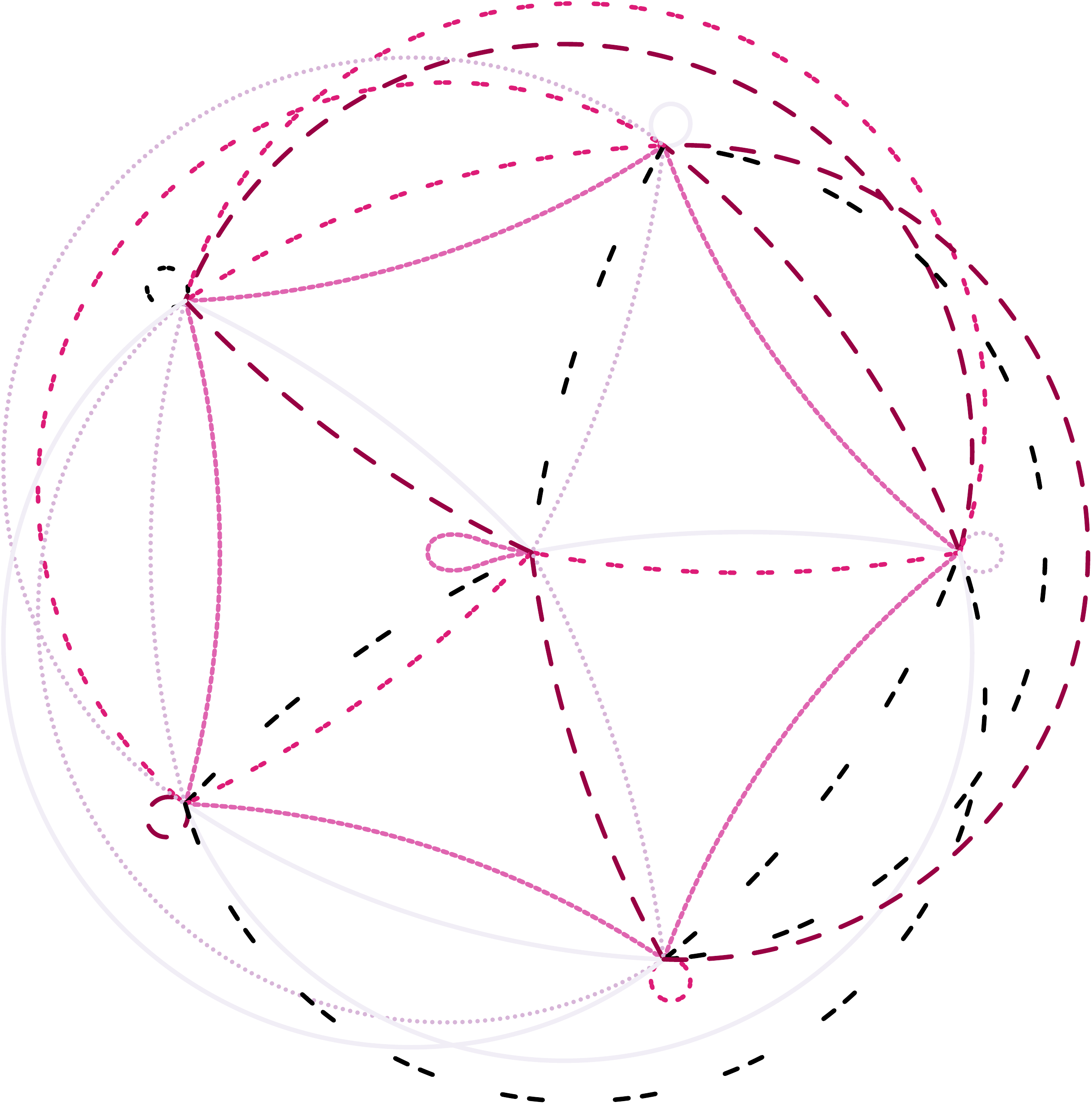}
    \end{tabular}
  \end{center}
	\caption{Left: gluing orbits of face classes from $\mathcal{D}_{\WS}$ in $6$-sheeted Haken cover $\mathcal{C}$. Right: face pairing graph of $\mathcal{C}$. Colours encode face classes in the base $\mathcal{D}_{\WS}$. Note that each dodecahedron has one self-identification and, in particular, that the cover is not cyclic. \label{fig:fpg}}
\end{figure}

The six surfaces consist of 60 mid-cubes each. All surfaces can be decomposed into $12$ ``pentagonal disks'' of five quadrilaterals each, which are parallel to one of the pentagonal faces of the complex, but slightly pushed into one of the adjacent dodecahedra. The six surfaces are given by their pentagonal disks and listed below. Since all of their vertices (which are intersections of the edges of the dodecahedra) must have degree $5$, each surface must have $12$ such vertices, $12$ pentagonal disks, and $30$ edges of pentagonal disks, and thus is of Euler characteristic $-6$. Moreover, since the Weber-Seifert space is orientable and the surface is 2--sided, it must be orientable of genus $4$. 

Every pentagonal disk is denoted by the corresponding pentagonal face it is parallel to, and the index of the dodecahedron it is contained in. The labelling follows Figure~\ref{fig:dodecahedron}.

$$
\begin{array}{llllllll}
S_1&=\Big\langle&(z,\odot)_1,&(y,\otimes)_1,&(v,\odot)_2,&(y,\odot)_2,&(w,\odot)_3,&(w,\otimes)_3, \\
&&(x,\otimes)_4,&(z,\otimes)_4,&(v,\otimes)_5,&(u,\otimes)_5,&(u,\odot)_6,&(x,\odot)_6\quad\Big\rangle
\end{array}
$$
$$
\begin{array}{llllllll}
S_2&=\Big\langle&(w,\otimes)_1,&(x,\odot)_1,&(u,\otimes)_2,&(y,\otimes)_2,&(u,\odot)_3,&(v,\odot)_3,\\
&&(w,\odot)_4,&(y,\odot)_4,&(z,\otimes)_5,&(z,\odot)_5,&(x,\otimes)_6,&(v,\otimes)_6\quad\Big\rangle
\end{array}
$$
$$
\begin{array}{llllllll}
S_3&=\Big\langle&(w,\odot)_1,&(v,\otimes)_1,&(w,\otimes)_2,&(z,\otimes)_2,&(x,\otimes)_3,&(u,\otimes)_3,\\
&&(z,\odot)_4,&(u,\odot)_4,&(x,\odot)_5,&(v,\odot)_5,&(y,\otimes)_6,&(y,\odot)_6\quad\Big\rangle
\end{array}
$$
$$
\begin{array}{llllllll}
S_4&=\Big\langle&(x,\otimes)_1,&(y,\odot)_1,&(w,\odot)_2,&(u,\odot)_2,&(x,\odot)_3,&(z,\odot)_3,\\
&&(v,\odot)_4,&(v,\otimes)_4,&(w,\otimes)_5,&(y,\otimes)_5,&(z,\otimes)_6,&(u,\otimes)_6\quad\Big\rangle
\end{array}
$$
$$
\begin{array}{llllllll}
S_5&=\Big\langle&(u,\otimes)_1,&(u,\odot)_1,&(v,\otimes)_2,&(z,\odot)_2,&(z,\otimes)_3,&(y,\odot)_3,\\
&&(x,\odot)_4,&(y,\otimes)_4,&(x,\otimes)_5,&(w,\odot)]_5,&(w,\otimes)_6,&(v,\odot)_6\quad\Big\rangle
\end{array}
$$
$$
\begin{array}{llllllll}
S_6&=\Big\langle&(v,\odot)_1,&(z,\otimes)_1,&(x,\otimes)_2,&(x,\odot)_2,&(v,\otimes)_3,&(y,\otimes)_3,\\
&&(u,\otimes)_4,&(w,\otimes)_4,&(y,\odot)_5,&(u,\odot)_5,&(w,\odot)_6,&(z,\odot)_6\quad\Big\rangle
\end{array}
$$

Note that the $12$ pentagonal disks of every surface component intersect each dodecahedron exactly twice. (A priori, given a $6$-fold cover of $\mathcal{D}_{\WS}$ with $6$ embedded surface components, such an even distribution is not clear: an embedded surface can intersect a dodecahedron in up to three pentagonal disks.) Moreover, every surface component can be endowed with an orientation, such that all of its dual edges point towards the centre of a dodecahedron. Hence, all surface components must be self-osculating through the centre points of some dodecahedron. 

\begin{remark}
  \label{rem:self}
  The fact that there must be some self-osculating surface components in $\mathcal{C}$ can also be deduced from the fact that the cover features self-identifications (i.e., loops in the face pairing graph). To see this, assume w.l.o.g. that the top and the bottom of a dodecahedron are identified. Then, for instance, pentagonal disk $P_{1}$ (which must be part of some surface component) intersecting the innermost pentagon in edge $\langle v_{1}, v_{2} \rangle$ must also intersect the dodecahedron in pentagon $P_{2}$, and the corresponding surface component must self-osculate (see Figure~\ref{fig:self}).

\begin{figure}[htb]
    \begin{center}
	\includegraphics[height=8cm]{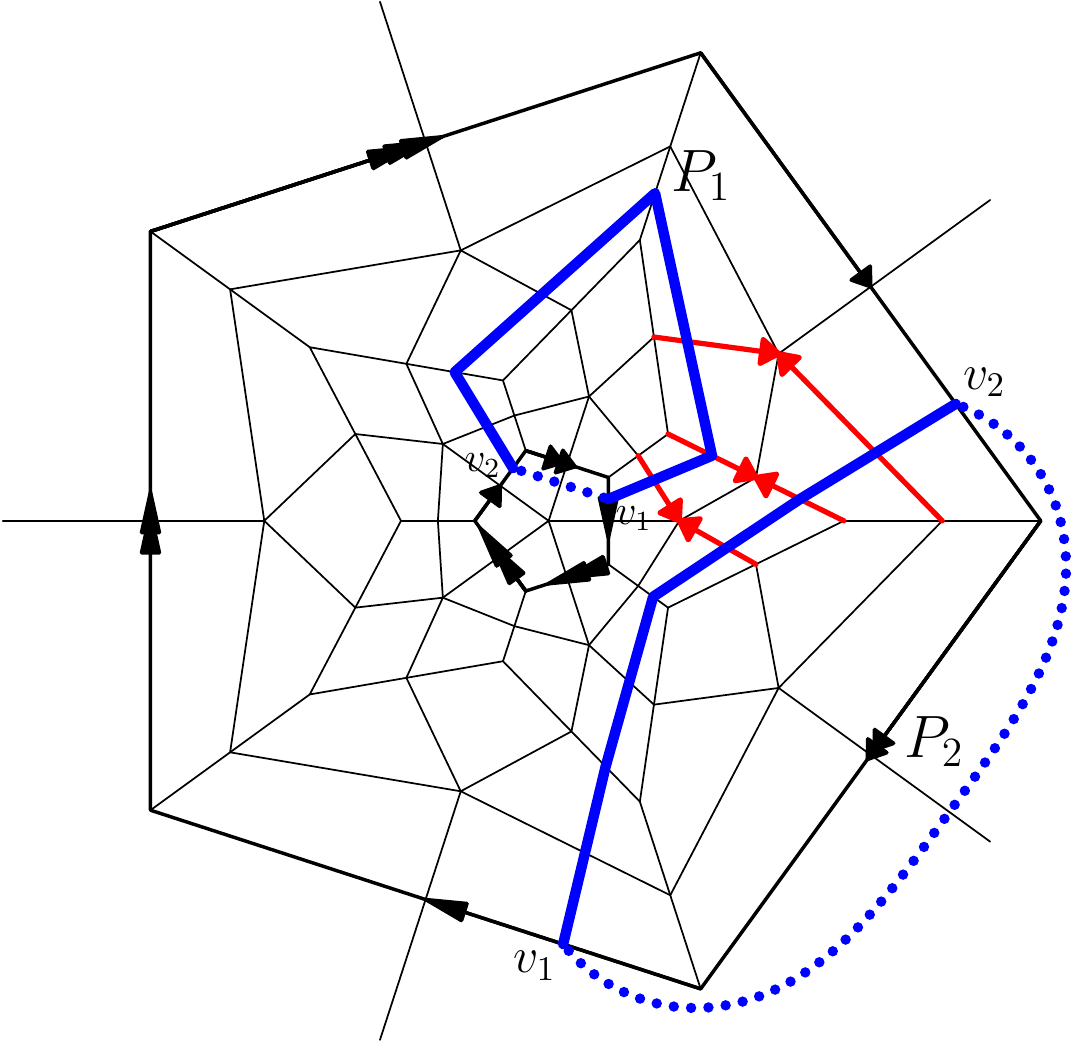}
    \end{center}
    \caption{Self-identifications (as indicated by the arrows) always result in a self-osculating component of the canonical surface -- as indicated by the two pentagonal disks $P_1$ and $P_2$ of the canonical surface, glued along the dotted edge $\langle v_1, v_2 \rangle$ as a result of the self-identification. \label{fig:self}}
\end{figure}
\end{remark}


\subsection{A special cover of degree $60$}
\label{ssec:special}

The (non-trivial) normal cores of the subgroups of index up to $6$ are all of index either $60$ or $360$ in $\pi_1 (\WS)$. 
For the computation of normal cores we use the \texttt{GAP} function \texttt{Core} \cite{GAP}.
One of the index $60$ subgroups is the index $12$ normal core of Hempel's cover mentioned in Section~\ref{ssec:five}.
This is equal to the index $10$ normal core of $\pi_1 (\mathcal{C})$ from Section~\ref{ssec:six}, and we now show that it produces  a special cover $\mathcal{S}$ of $\WS$ of degree $60$. The deck transformation group is the alternating group $\operatorname{A}_5$ and the abelian invariant of the cover is $\mathbb{Z}^{41} \oplus \mathbb{Z}_2^{12}$.

The generators of $\pi_1 (\mathcal{S})$ are

$$
\begin{array}{llllll}
   u v^{-1} w^{-1} , & u w^{-1} x^{-1} , & u x^{-1} y^{-1} , & u y^{-1} z^{-1} , & u z^{-1} v^{-1} , & u^{-1} v z \\
  u^{-1} w v , & u^{-1} x w , & u^{-1} y x , & u^{-1} z y , & v u y^{-1} u^{-1} , & v u^{-1} w \\
   v w y^{-1} , & v x v^{-1} u^{-1} , & v x^{-1} z , & v y^{-1} x^{-1} , & v^{-1} u z^{-1} , & v^{-1} u^{-1} x u \\
  v^{-1} x y , & v^{-1} y^{-1} v u , & w u z^{-1} u^{-1} , & w u^{-1} x , & w v^{-1} x v , &  w x z^{-1} \\
  w y w^{-1} u^{-1} , & w z w^{-1} v , & w z^{-1} y^{-1} , & w^{-1} u^{-1} y u , & w^{-1} v z v^{-1} , & w^{-1} x w v^{-1} \\
  w^{-1} z^{-1} w u , & x u v^{-1} u^{-1} , & x v w v^{-1} , &  x w^{-1} y w , & x z x^{-1} u^{-1} , & x^{-1} u^{-1} z u \\
  x^{-1} v^{-1} x u , & x^{-1} w x v , & x^{-1} y x w^{-1} , & y u w^{-1} u^{-1} , & y v y^{-1} u^{-1} , & y w x w^{-1} \\
  y x^{-1} z x , &  y^{-1} u^{-1} v u , & y^{-1} v^{-1} z^{-1} v , & y^{-1} w^{-1} y u , & y^{-1} x y w , & z u x^{-1} u^{-1} \\
  z^{-1} u^{-1} w u , & u^5 , & u^2 v^{-1} w^{-1} u^{-1} , & u v u y^{-1} u^{-2} , &  u v w^{-1} x^{-2} , & u v x v^{-1} u^{-2} \\
  u v y v^{-2} , & u w u z^{-1} u^{-2} , & u w x^{-1} y^{-2} , & u w z w^{-2} , & u x u v^{-1} u^{-2} , & u x y^{-1} z^{-2} \\
  u y u w^{-1} u^{-2} , &  u^{-2} v z u , & u^{-2} v^{-1} w^{-1} u^{-2} , & u^{-1} v^{-1} u^{-1} y^{-1} u^{-2} , & u^{-1} v^{-1} x^{-1} v^2 .  & 
\end{array}
$$

In order to see that $\mathcal{S}$ is in fact a special cover, we must establish a number of observations on embedded surface components in covers of $\hat{\mathcal{D}}_{\WS}$. 
In the following paragraphs we always assume that we are given a finite cover $\mathcal{B}$ of $\mathcal{D}_{\WS}$ together with its canonical immersed surface defined by the lift of $\hat{\mathcal{D}}_{\WS}$ in $\mathcal{B}$. Whenever we refer to faces of the decomposition of $\mathcal{B}$ into dodecahedra, we explicitly say so. Otherwise we refer to the faces of the lift of the natural cubulation in $\mathcal{B}$. We start with a simple definition.

\begin{definition}
  A dodecahedral vertex is said to be \emph{near} a component $S$ of the canonical immersed surface of $\mathcal{B}$ if it is the endpoint of an edge of the cubulation dual to $S$.
\end{definition}

\begin{lemma}
  \label{lem:selfosc}
  An embedded component $S$ of the canonical immersed surface of $\mathcal{B}$ self-osculates if and only if at least one of the following two situations occurs.

  \begin{enumerate}[a)]
    \item There exists a dodecahedron containing more than one pentagonal disk of $S$.
    \item The number of dodecahedral vertices near $S$ is strictly smaller than its number of pentagonal disks.
  \end{enumerate}
\end{lemma}

\begin{proof}
  First note that $S$ is $2$--sided and can be transversely oriented such that one side always points towards the centres of the dodecahedra it intersects. From this it is apparent that if one of a) or b) occurs, then the surface component must self-osculate.

  Assume that a) does not hold; that is, all dodecahedra contain at most one pentagonal disk of $S$. Hence, no self-osculation can occur through the centre of a dodecahedron. Since every surface component $S$ is made out of pentagonal disks, with five of such disks meeting in every vertex, $S$ has as many pentagonal disks as it has pentagonal vertices. Moreover, every such pentagonal vertex of $S$ must be near exactly one dodecahedral vertex of $\mathcal{B}$. Hence, the number of dodecahedral vertices that $S$ is near to is bounded above by its number of pentagonal disks. Equality therefore occurs if and only if $S$ is not near any dodecahedral vertex twice. Hence, if b) does not hold, no self-osculation can occur through a vertex of a dodecahedron.

  \medskip

It remains to prove that if $S$ self-osculates, then it must self-osculate through a centre point of a dodecahedron or through a vertex of a dodecahedron. The only other possibilities are that it self-osculates through either the midpoint of a dodecahedral edge or through the centre point of a dodecahedral face.

First assume that the surface self-osculates through the midpoint of a dodecahedral edge $e$. Then either the surface has two disjoint pentagonal disks both parallel to $e$ and hence also self-osculates through the two dodecahedral endpoints of $e$; or the surface has two disjoint pentagonal disks both intersecting $e$, in which case there exists a pair of pentagonal disks in the same dodecahedron -- and the surface self-osculates through the centre of that dodecahedron.

Next assume the surface self-osculates through the centre point of a dodecahedral face $f$. Then either the surface has two disjoint pentagonal disks both parallel to $f$ and hence also self-osculates through the five dodecahedral vertices of $f$; or the surface has two disjoint pentagonal disks both intersecting $f$, in which case there exists a pair of pentagonal disks in the same dodecahedron and the surface self-osculates through the centre of that dodecahedron.
\end{proof}

\begin{lemma}
  \label{lem:interosc}
  A pair of intersecting, embedded, and non-self-osculating components $S$ and $T$ of the canonical immersed surface of $\mathcal{B}$ inter-osculates if and only if at least one of the following two situations occurs.

  \begin{enumerate}[a)]
    \item Some dodecahedron contains pentagonal disks of both $S$ and $T$ which are disjoint.
    \item The number of all dodecahedral vertices near $S$ or $T$ minus the number of all pairs of intersecting pentagonal disks is strictly smaller than the number of all pentagonal disks in $S$ or $T$.
  \end{enumerate}
\end{lemma}

\begin{proof}
  We first need to establish the following three claims.

  \noindent
  {\bf Claim 1}: If $S$ and $T$ inter-osculate, then they inter-osculate through the centre of a dodecahedron or a vertex of a dodecahedron.

  This follows from the arguments presented in the second part of the proof of Lemma~\ref{lem:selfosc} since inter-osculation locally behaves exactly like self-osculation.

  \noindent
  {\bf Claim 2}: Every pentagonal disk of $S$ intersects $T$ in at most one pentagonal disk and vice versa.

  A pentagonal disk can intersect another pentagonal disk in five different ways. Every form of multiple intersection causes either $S$ or $T$ to self-osculate or even self-intersect. 

  \noindent
  {\bf Claim 3}: A dodecahedral vertex near an intersection of $S$ and $T$ cannot be near any other pentagonal disk of $S$ or $T$, other than the ones close to the intersection.

  Assume otherwise, then this causes either $S$ or $T$ to self-osculate or even self-intersect.

  \medskip
  We now return to the proof of the main statement. If a) is satisfied, then the surface pair inter-osculates through the centre of the dodecahedron (see also the proof of Lemma~\ref{lem:selfosc}). If b) is satisfied, then by Claim 2 and Claim 3, both $S$ and $T$ must be near a dodecahedral vertex away from their intersections and thus $S$ and $T$ inter-osculate.
  
For the converse assume that neither a) nor b) holds. By Claim 1, it suffices to show that $S$ and $T$ do not  inter-osculate through the centre of a dodecahedron or a vertex of a dodecahedron.
  
We first show that $S$ and $T$ do not inter-osculate through the centre of a dodecahedron. If at most one of $S$ or $T$ meets a dodecahedron, then this is true for its centre. Hence assume that both $S$ and $T$ meet a dodecahedron in pentagonal discs. By Claim 2 the dodecahedron contains exactly one pentagonal disc from each surface. These intesect since a) is assumed false. The only dual edges to $S$ (resp. $T$) with a vertex at the centre of the cube run from the centre of the pentagonal face of the dodecahedron dual to $S$ (resp. $T$) to the centre of the dodecahedron. But these two edges lie in the boundary of a square in the dodecahedron since the pentagonal discs intersect and hence the pentagonal faces are adjacent. Hence $S$ and $T$ do not inter-osculate through the centre of a dodecahedron.
    
We next show that $S$ and $T$ do not inter-osculate through the vertex of a dodecahedron. The negation of b) is that the number of all dodecahedral vertices near $S$ or $T$ minus the number of all pairs of intersecting pentagonal disks equals the number of all pentagonal disks of $S$ and $T.$ Suppose a dodecahedral vertex is the endpoint of dual edges to squares in $S$ and $T$. If the dual edges are contained in the same dodecahedron then they are in the boundary of a common square. Hence assume they are contained in different dodecahedra. Then the equality forces at least one of the dual edges to be in the boundary of a cube intersected by both $S$ and $T$. But then at least one of the surfaces self-osculates. 
\end{proof}

Due to Lemmata~\ref{lem:selfosc} and \ref{lem:interosc}, checking for self-osculating embedded surface components is a straightforward task. Furthermore, as long as surface components are embedded and non-self-osculating, checking for inter-osculation of a surface pair is simple as well.

In the cover $\mathcal{S}$ we have:
\begin{enumerate}[a)]
  \item the canonical immersed surface splits into $60$ embedded components,
  \item every surface component of $\mathcal{S}$ is made up of $12$ pentagonal disks (and thus is orientable of genus $4$, see the description of the canonical surface components of $\mathcal{C}$ in Section~\ref{ssec:six} for details), 
  \item every surface component distributes its $12$ pentagonal disks over $12$ distinct dodecahedra, 
  \item every surface component is near $12$ dodecahedral vertices, and
  \item every pair of intersecting surface components intersects in exactly three pentagonal disks (and hence in exactly three dodecahedra), and for each such pair both surface components combined are near exactly $21$ dodecahedral vertices.
\end{enumerate}

These properties of $\mathcal{S}$ can be checked using the \texttt{GAP} script available from \cite{ST}. From them, and from Lemmata~\ref{lem:selfosc} and \ref{lem:interosc} it follows that $\mathcal{S}$ is a special cover. The gluing orbits for $\mathcal{S}$ of the face classes from $\mathcal{D}_{\WS}$, as well as all $60$ surface components are listed in Appendix~\ref{app:special}. 


\subsection{Covers of higher degree}

An exhaustive enumeration of all subgroups up to index $9$ reveals a total of $490$ covers, but no further examples of covers where the canonical surface splits into embedded components (and in particular no further special covers). There are, however, $20$ examples of degree $8$ covers where the canonical surface splits into two immersed connected components (all with first homology group $\mathbb{Z} \oplus \mathbb{Z}_{2}^{3} \oplus \mathbb{Z}_{3}^{2} \oplus \mathbb{Z}_{5}^{3}$). Moreover, there are $10$ examples of degree $9$ covers, where the canonical surface splits into two components, one of which is embedded (all with first homology group $\mathbb{Z} \oplus \mathbb{Z}_{3} \oplus \mathbb{Z}_{4} \oplus \mathbb{Z}_{5}^{3} \oplus \mathbb{Z}_{7}$). All of them are Haken, as can be seen by their first integral homology groups.

\begin{table}[htbp]
    \begin{center}
	\begin{tabular}{|c|l|c|c|c|r|}
		\toprule
		degree & $H_1(X)$ & $\beta_1(X)$ & $\#$ surf. comp. & $\exists$ emb. surf. comp. & $\# $ of covers \\
		\midrule
		$1$&$\mathbb{Z}_{5}^{3}$ & $0$ & $1$ & no & $1$ \\
		\midrule
		&&&&& {\bf $\Sigma$ = 1} \\
		\midrule
		\midrule

		$5$&$\mathbb{Z}_{5}^{2} \oplus \mathbb{Z}_{25}^{2}$ &$0$ & $1$ & no & $25$ \\
		&$\mathbb{Z}_{3} \oplus \mathbb{Z}_{5} \oplus \mathbb{Z}_{25}^{3}$ &$0$ & $1$ & no & $6$ \\
		&$\mathbb{Z}_{5}^{6} \oplus \mathbb{Z}_{25}$ &$0$ & $1$ & no & $6$ \\
		&$\mathbb{Z}^{4} \oplus \mathbb{Z}_{3}^{2}$ &{\bf 4} & $5$ & no & $1$ \\

		\midrule
		&&&&& {\bf $\Sigma$ = 38} \\
		\midrule
		\midrule

		$6$&$\mathbb{Z}_{4} \oplus \mathbb{Z}_{5}^{3} $ &$0$ & $1$ & no & $6$ \\
		&$\mathbb{Z}_{3}^{2} \oplus \mathbb{Z}_{4} \oplus \mathbb{Z}_{5}^{3} $ &$0$ & $1$ & no & $15$ \\
		&$\mathbb{Z}_{3} \oplus \mathbb{Z}_{5}^{3} \oplus \mathbb{Z}_{11}^{2} $ &$0$ & $1$ & no & $24$ \\
		&$\mathbb{Z}_{3}^{2} \oplus \mathbb{Z}_{4} \oplus \mathbb{Z}_{5}^{3} \oplus \mathbb{Z}_{16}^{2} $ &$0$ & $1$ & no & $15$ \\
		&$\mathbb{Z}^{5} \oplus \mathbb{Z}_{2}^{2} \oplus \mathbb{Z}_{5}^{3}$ & {\bf 5} & $6$ & yes (all) & $1$ \\
		\midrule
		&&&&& {\bf $\Sigma$ = 61} \\
		\midrule
		\midrule

		$7$ & $\mathbb{Z}_{2}^{3} \oplus \mathbb{Z}_{5}^{3}$ & $0$ & $1$ & no &$20$ \\
		& $\mathbb{Z}_{3}^{2} \oplus \mathbb{Z}_{5}^{3} \oplus \mathbb{Z}_{7} \oplus \mathbb{Z}_{9} \oplus \mathbb{Z}_{11}$ &$0$ & $1$ & no & $30$ \\
		\midrule
		&&&&& {\bf $\Sigma$ = 50} \\
		\midrule
		\midrule

		$8$ & $\mathbb{Z}_{2}^{3} \oplus \mathbb{Z}_{5}^{3}$ &$0$ & $1$ & no & $40$ \\
		& $\mathbb{Z}_{2}^{3} \oplus \mathbb{Z}_{3} \oplus \mathbb{Z}_{5}^{3} \oplus \mathbb{Z}_{9}$ &  $0$ &$1$ & no &$20$ \\
		& $\mathbb{Z}_{2} \oplus \mathbb{Z}_{5}^{3} \oplus \mathbb{Z}_{7}^{3}$ &$0$ & $1$ & no & $40$ \\
		& $\mathbb{Z}_{3} \oplus \mathbb{Z}_{4} \oplus \mathbb{Z}_{5}^{3} \oplus \mathbb{Z}_{19}^{2}$ &$0$ & $1$ & no & $15$ \\
		& $\mathbb{Z}_{2}^{5} \oplus \mathbb{Z}_{3} \oplus \mathbb{Z}_{5}^{3} \oplus \mathbb{Z}_{7}^{2}$ &$0$ & $1$ & no & $10$ \\
		& $\mathbb{Z} \oplus \mathbb{Z}_{2}^{3} \oplus \mathbb{Z}_{3}^{2} \oplus \mathbb{Z}_{5}^{3}$ &{\bf 1} &  $2$ & no &$20$ \\
		& $\mathbb{Z} \oplus \mathbb{Z}_{2}^{3} \oplus \mathbb{Z}_{3} \oplus \mathbb{Z}_{5}^{3} \oplus \mathbb{Z}_{13}^{2}$ &{\bf 1} & $1$ & no & $40$ \\
		\midrule
		&&&&& {\bf $\Sigma$ =  185} \\
		\midrule
		\midrule

		$9$& $\mathbb{Z}_{2} \oplus \mathbb{Z}_{3} \oplus \mathbb{Z}_{4}^{2} \oplus \mathbb{Z}_{5}^{3}$ &$0$ & $1$ & no & $60$ \\
		& $\mathbb{Z}_{2} \oplus \mathbb{Z}_{3} \oplus \mathbb{Z}_{5}^{3} \oplus \mathbb{Z}_{8}^{2}$ &$0$ & $1$ & no & $40$ \\
		& $\mathbb{Z}_{2}^{4} \oplus \mathbb{Z}_{5}^{3} \oplus \mathbb{Z}_{9} \oplus \mathbb{Z}_{89}$ &$0$ & $1$ & no & $15$ \\
		& $\mathbb{Z} \oplus \mathbb{Z}_{3} \oplus \mathbb{Z}_{4} \oplus \mathbb{Z}_{5}^{3} \oplus \mathbb{Z}_{7}$ &{\bf 1} & $2$ & yes (one) &$10$ \\
		& $\mathbb{Z} \oplus \mathbb{Z}_{3} \oplus \mathbb{Z}_{4}^{2} \oplus \mathbb{Z}_{5}^{3} \oplus \mathbb{Z}_{19}$ &{\bf 1} & $1$ & no & $30$ \\
		\midrule
		&&&&& {\bf $\Sigma$ = 155} \\
                \midrule
		\bottomrule
	\end{tabular}
	\caption{First homology groups of all $490$ covers of degree up to nine. \label{tab:1to9}}
    \end{center}
\end{table}

\begin{table}[htbp]
    \begin{center}
	\begin{tabular}{|l|r|r|}
		\toprule
		$H_1(X)$ & $\#$ surf. comp. & $\#$ covers \\
		\midrule
		$\mathbb{Z}^{5} \oplus \mathbb{Z}_{4}^{2} \oplus \mathbb{Z}_{5}^{3}$&$6$&$20$\\
		$\mathbb{Z}^{5} \oplus \mathbb{Z}_{2}^{2} \oplus \mathbb{Z}_{4} \oplus \mathbb{Z}_{5}^{3}$&$6$&$20$\\
		$\mathbb{Z}^{6} \oplus \mathbb{Z}_{2}^{2} \oplus \mathbb{Z}_{5}^{3}$&$7$&$12$\\
		$\mathbb{Z}^{7} \oplus \mathbb{Z}_{5}^{3}$&$7$&$12$\\
		\midrule
		&& {\bf $\Sigma$ =  64} \\
          \midrule
          \bottomrule
	\end{tabular}

	\caption{Summary of all $64$ fix-point free double covers of $\mathcal{C}$. By construction, all surface components are embedded. \label{tab:higher}}
    \end{center}
\end{table}

\medskip
In an attempt to obtain further special covers we execute a non-exhaustive, heuristic search for higher degree covers. This is necessary since complete enumeration of subgroups quickly becomes infeasible for subgroups of index larger than $9$. This more targeted search is done in essentially two distinct ways.

\medskip
In the first approach we compute normal cores of all irregular covers of degrees $7$, $8$, and $9$ from the enumeration of subgroups of $\pi_1 (\WS)$ of index at most $9$ described above. This is motivated by the fact that the index $60$ normal core of $\pi_1 (\mathcal{C})$ yields a special cover. The normal cores have indices $168$, $504$, $1344$, $2520$, $20160$, and $181440$. Of the ones with index at most $2520$, we construct the corresponding cover. Very often, the covers associated to these normal cores exhibit a single (immersed) surface component. However, the normal cores of the $10$ subgroups corresponding to the covers of degree $9$ with two surface components yield (regular) covers where the canonical immersed surface splits into nine embedded components. All of these covers are of degree $504$ with deck transformation group $\operatorname{PSL}(2,8)$. Each of the surface components has $672$ pentagons. Accordingly, each of them must be (orientable) of genus $169$. All nine surface components necessarily self-osculate (they are embedded and contain more pentagonal disks than there are dodecahedra in the cover). The first homology group of all of these covers is given by
$$\mathbb{Z}^{8} \oplus \mathbb{Z}_2^{10} \oplus \mathbb{Z}_3 \oplus \mathbb{Z}_4^{9} \oplus \mathbb{Z}_5^{17} \oplus \mathbb{Z}_7 \oplus \mathbb{Z}_8^{6} \oplus \mathbb{Z}_9^{7} \oplus \mathbb{Z}_{17}^{28} \oplus \mathbb{Z}_{27}^{7} \oplus \mathbb{Z}_{29}^{9} \oplus \mathbb{Z}_{83}^{18} .$$
In addition, there are $120$ subgroups with a core of order $1,344$, and factor group isomorphic to a semi-direct product of $\mathbb{Z}_2^3$ and $\operatorname{PSL}(3,2)$. For $40$ of them the corresponding (regular) cover splits into $8$ immersed components. These include the covers of degree $8$ where the canonical immersed surface splits into two immersed components.

\medskip
In the second approach we analyse low degree covers of $\mathcal{C}$ from Section~\ref{ssec:six}. This is motivated by the fact that, in such covers, the canonical surface necessarily consists of embedded components. 

There are $127$ $2$-fold covers of $\mathcal{C}$, $64$ of which are fix-point free (i.e., they do not identify two pentagons of the same dodecahedron -- a necessary condition for a cover to be special, see the end of Section~\ref{ssec:six}). For $40$ of them the canonical surface still only splits into six embedded components. For the remaining $24$, the surface splits into $7$ components. For more details, see Table~\ref{tab:higher}.

The $127$ $2$-fold covers of $\mathcal{C}$ altogether have $ 43,905 $ $2$-fold covers. Amongst these $24$-fold covers of $\mathcal{D}_{\WS}$, $ 16,192 $ are fix-point free. They admit $6$ to $14$ surface components with a single exception where the surface splits into $24$ components. This cover is denoted by $\mathcal{E}$. Details on the number of covers and surface components can be found in Table~\ref{tab:24}.

We have for the generators of the subgroup corresponding to the cover $\mathcal{E}$

$$ \begin{array}{lllllll} 
    u^{-2},&
    uvz,&
    uv^{-1}w^{-1},&
    uwv,&
    uw^{-1}x^{-1},&
    uxw,&
    ux^{-1}y^{-1}, \\

    uyx,&
    uy^{-1}z^{-1},&
    uzy,&
    uz^{-1}v^{-1},&
    vux,&
    vu^{-1}w,&
    vwy^{-1}, \\

    vxz,&
    vx^{-1}z,&
    vy^{-1}x^{-1},&
    z^{-1}uy^{-1},&
    z^{-1}u^{-1}x^{-1},&
    z^{-1}vy^{-1}v^{-1},&
    z^{-1}wx, \\
    
    z^{-1}w^{-1}zv^{-1},&
    z^{-1}yv^{-2},&
    z^{-1}y^{-1}w,&
    z^{-2}wv^{-1},&
    wuy,&
    wv^{-1}xv .& 
  \end{array} $$

Surface components in $\mathcal{E}$ are small ($12$ pentagonal disks per surface, as also observed in the degree $60$ special cover $\mathcal{S}$, see Section~\ref{ssec:special}). This motivates an extended search for a degree $48$ special cover by looking at degree $2$ covers of $\mathcal{E}$. However, amongst the $131,071$ fix-point free covers of degree $2$, no special cover exists. More precisely, there are $120,205$ covers with $24$ surface components, $10,200$ with $25$ surface components, $240$ with $26$ and $27$ surface components each, $162$ with $28$, and $24$ with $33$ surface components. For most of them, most surface components self-osculate.

\begin{table}[!h]
  \begin{center}
    \begin{tabular}{|l|r|r|r|r|r|r|r|r|r|r|r|r||r|}
    \toprule
    $\#$ surf. comp. & $6$&$7 $&$7 $&$8 $&$8 $&$8 $&$9 $&$9 $&$12 $&$14 $&$14 $&$24 $&$ $ \\
    \midrule
    $\#$ covers & $ 8,960$&$ 3,240$&$ 2,160$&$ 180$&$ 720$&$ 540$&$ 240$&$ 24$&$ 85$&$ 24$&$ 18$&$ 1$& {\bf $\Sigma$ = 16,192 } \\
    \bottomrule

    \end{tabular}
  \end{center}
  \caption{Summary of all $16,192 $ fix-point free double covers of the $127$ double covers of $\mathcal{C}$. By construction, all surface components are embedded. \label{tab:24}} 
\end{table}


\section{Poincar\'e homology sphere and projective space}
\label{sec:covers of PHS}

The Poincar\'e homology sphere has as fundamental group the binary icosahedral group of order $120$, which is isomorphic to $\operatorname{SL}(2,5)$. From its subdivision given by the dodecahedron, we can deduce a presentation with six generators dual to the six pentagons of the subdivision, and one relator dual to each of the $10$ edges:

\begin{figure}[htb]
	\centerline{\includegraphics[height=6cm]{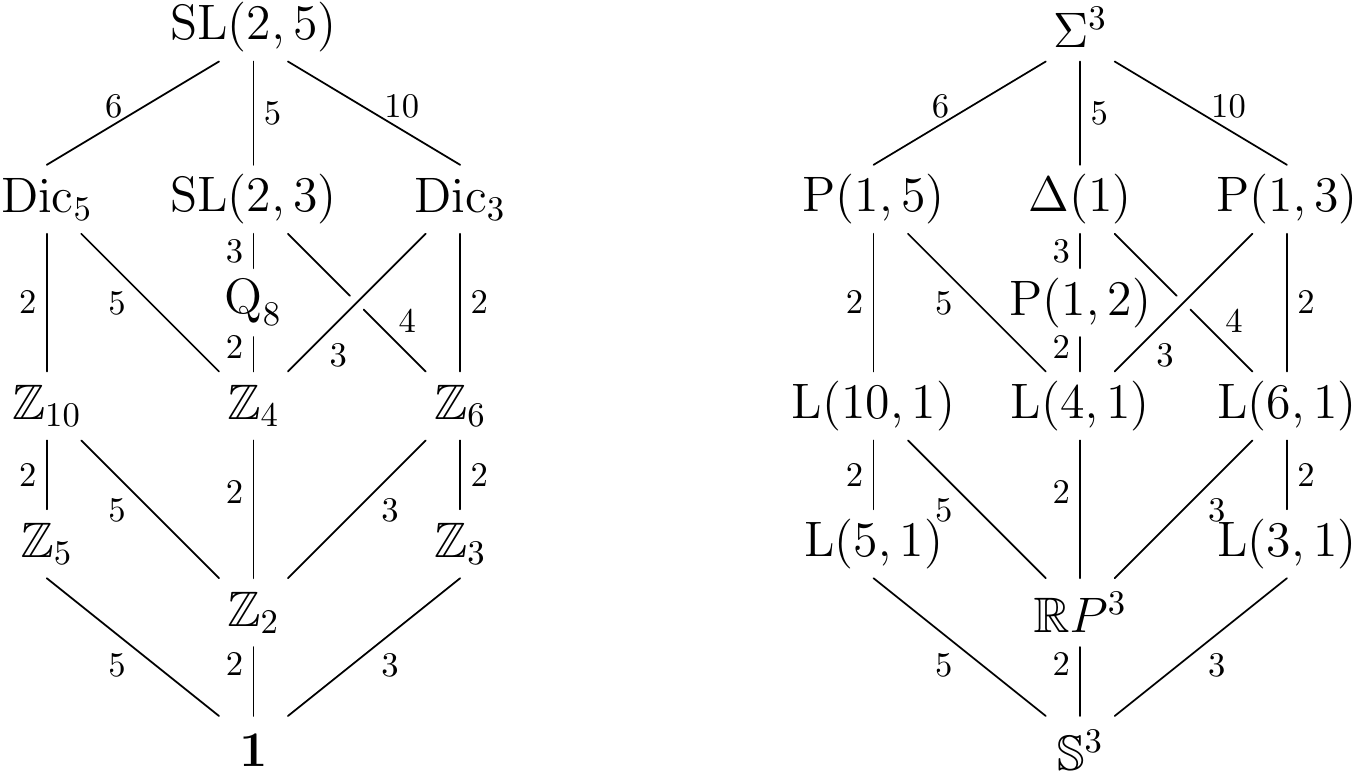}}
	\caption{Left: subgroup lattice of $\operatorname{SL}(2,5)$ with indices. Right: covers of $\PSH$ with degrees. Here $\operatorname{P}(n,m)$ denotes the prism space with parameters $n$ and $m$, and $\Delta(n)$ denotes the tetrahedral space with parameter $n$.\label{fig:sl25}}
\end{figure}

$$ \begin{array}{llllll} \pi_1 (\mathcal{D}_{\PSH}) = \langle \,\, u,\, v,\, w,\, x,\, y,\, z \, \mid & 
	uxz,&uyv,&uzw,&uvx,&uwy, \\
	&xy^{-1}z,&yz^{-1}v,&zv^{-1}w,&vw^{-1}x,&wx^{-1}y \,\, \rangle .
  \end{array} $$

$\operatorname{SL}(2,5)$ has $76$ subgroups falling into $12$ conjugacy classes forming the subgroup lattice shown in Figure~\ref{fig:sl25} on the left hand side. For the corresponding hierarchy of covers together with the topological types of the covering $3$-manifolds see Figure~\ref{fig:sl25} on the right hand side.

By construction, the universal cover of $\mathcal{D}_{\PSH}$ is the $120$-cell, which is dual to the simplicial $600$-cell. In particular, the dual cell decomposition of any of the $12$ covers is a (semi-simplicial) triangulation. The dual of $\mathcal{D}_{\PSH}$ itself is isomorphic to the minimal five-tetrahedron triangulation of the Poincar\'e homology sphere.

\begin{table}[b!]
    \begin{center}
	\begin{tabular}{rccrcrcc}
            \toprule
	    deg. & top. type. & subgroup & $f$-vec. & embedded & $\#$ surf. & regular & deck trafo grp. \\
	    \midrule
            $1$&$\PSH$&$\operatorname{SL} (2,5)$&$(5,10,6,1)$&no&$1$&yes&$1$ \\
            $5$&$\Delta(1)$&$\operatorname{SL}(2,3)$&$(25,50,30,5)$&yes&$5$&no&$\operatorname{A}_5$ \\
            $6$&$\operatorname{P}(1,5)$&$\operatorname{Dic}(5)$&$(30,60,36,6)$&no&$6$&no&$\operatorname{A}_5$ \\
            $10$&$\operatorname{P}(1,3)$&$\operatorname{Dic}(3)$&$(50,100,60,10)$&yes&$10$&no&$\operatorname{A}_5$ \\
            $12$&$\operatorname{L}(10,1)$&$\mathbb{Z}_{10}$&$(60,120,72,12)$&no&$12$&no&$\operatorname{A}_5$ \\
            $15$&$\operatorname{P}(1,2)$&$Q_8$&$(75,150,90,15)$&yes&$15$&no&$\operatorname{A}_5$ \\
            $20$&$\operatorname{L}(6,1)$&$\mathbb{Z}_6$&$(100,200,120,20)$&yes&$20$&no&$\operatorname{A}_5$ \\
            $24$&$\operatorname{L}(5,1)$&$\mathbb{Z}_5$&$(120,240,144,24)$&yes&$24$&no&$\operatorname{SL} (2,5)$ \\
            $30$&$\operatorname{L}(4,1)$&$\mathbb{Z}_4$&$(150,300,180,30)$&yes&$30$&no&$\operatorname{A}_5$ \\
            $40$&$\operatorname{L}(3,1)$&$\mathbb{Z}_3$&$(200,400,240,40)$&yes&$40$&no&$\operatorname{SL} (2,5)$ \\
            $60$&$\mathbb{R}P^3$&$\mathbb{Z}_2$&$(300,600,360,60)$&yes&$60$&yes&$\operatorname{A}_5$ \\
            $120$&$\mathbb{S}^3$&$\mathbf{1}$&$(600,1200,720,120)$&yes&$120$&yes&$\operatorname{SL} (2,5)$ \\
            \bottomrule
	\end{tabular}
    \end{center}
    \caption{Covers of $\PSH$. $\operatorname{P}(n,m)$ denotes the prism space with parameters $n$ and $m$, and $\Delta(n)$ denotes the tetrahedral space with parameter $n$. ``$f$-vec.'' denotes the $f$-vector of the cover as a decomposition into dodecahedra. I.e., $(25,50,30,5)$ means that the corresponding cover contains $25$ vertices, $50$ edges, $30$ pentagons, and $5$ dodecahedra. ``deck trafo grp.'' denotes the deck transformation group of the cover.\label{tab:PSH}}
\end{table}

Most of the topological types are determined by the isomorphism type of the subgroups. The only two non-trivial cases are the lens spaces $\operatorname{L} (5,1)$ and $\operatorname{L} (10,1)$. For the former, we passed to the dual triangulation of the cover of degree $24$ using the {\em GAP}-package simpcomp \cite{simpcomp}, and then fed the result to the $3$-manifold software {\em Regina} \cite{Regina} to determine the topological type of the cover to be $\operatorname{L} (5,1)$. The latter is then determined by the observation that there is no $2$-to-$1$-cover of $\operatorname{L} (10,3)$ to $\operatorname{L} (5,1)$.

Regarding the canonical immersed surface, the situation is quite straightforward. Since all edges of $\mathcal{D}_{\PSH}$, or of any of its covers, are of degree three, the canonical surface is a surface decomposed into pentagonal disks with three such disks meeting in each vertex. Consequently, all surface components must be $2$-spheres isomorphic to the dodecahedron, thus have $12$ pentagons, and the number of connected components of the canonical surface must coincide with the degree of the cover. Moreover, each surface component runs parallel to the $2$-skeleton of a single dodecahedron, and the surface components are embedded if and only if there are no self-intersections of dodecahedra.

In more detail the relevant properties of all covers are listed in Table~\ref{tab:PSH}.

The case of the projective space is rather simple. The only proper cover (of degree $>1$) is the universal cover of degree $2$. Since the edges of $\mathcal{D}_{\mathbb{R}P^3}$ are all of degree two, the canonical surface of $\mathcal{D}_{\mathbb{R}P^3}$ has six embedded sphere components, each  consisting of two pentagons glued along their boundary, surrounding one of the six pentagonal faces each. Consequently, the universal cover is a $3$-sphere decomposed into two balls along a dodecahedron with the canonical surface splitting into $12$  sphere components.



\address{Discrete Geometry group, Mathematical Institute, \\
Freie Universit\"at Berlin, \\Arnimallee 2, 14195 Berlin, Germany\\ {jonathan.spreer@fu-berlin.de\\-----}}

\address{School of Mathematics and Statistics F07,\\ The University of Sydney,\\ NSW 2006 Australia\\ {stephan.tillmann@sydney.edu.au}}

\Addresses

\newpage

\appendix

\section{The special cover $\mathcal{S}$}
\label{app:special}

The special cover $\mathcal{S}$ from Section~\ref{ssec:special} is of degree $60$ with deck transformation group $\operatorname{A}_5$ and abelian invariant $\mathbb{Z}^{41} \oplus \mathbb{Z}_2^{12}$. The subgroup is generated by

$$
\begin{array}{llllll}
   u v^{-1} w^{-1} , & u w^{-1} x^{-1} , & u x^{-1} y^{-1} , & u y^{-1} z^{-1} , & u z^{-1} v^{-1} , & u^{-1} v z \\
  u^{-1} w v , & u^{-1} x w , & u^{-1} y x , & u^{-1} z y , & v u y^{-1} u^{-1} , & v u^{-1} w \\
   v w y^{-1} , & v x v^{-1} u^{-1} , & v x^{-1} z , & v y^{-1} x^{-1} , & v^{-1} u z^{-1} , & v^{-1} u^{-1} x u \\
  v^{-1} x y , & v^{-1} y^{-1} v u , & w u z^{-1} u^{-1} , & w u^{-1} x , & w v^{-1} x v , &  w x z^{-1} \\
  w y w^{-1} u^{-1} , & w z w^{-1} v , & w z^{-1} y^{-1} , & w^{-1} u^{-1} y u , & w^{-1} v z v^{-1} , & w^{-1} x w v^{-1} \\
  w^{-1} z^{-1} w u , & x u v^{-1} u^{-1} , & x v w v^{-1} , &  x w^{-1} y w , & x z x^{-1} u^{-1} , & x^{-1} u^{-1} z u \\
  x^{-1} v^{-1} x u , & x^{-1} w x v , & x^{-1} y x w^{-1} , & y u w^{-1} u^{-1} , & y v y^{-1} u^{-1} , & y w x w^{-1} \\
  y x^{-1} z x , &  y^{-1} u^{-1} v u , & y^{-1} v^{-1} z^{-1} v , & y^{-1} w^{-1} y u , & y^{-1} x y w , & z u x^{-1} u^{-1} \\
  z^{-1} u^{-1} w u , & u^5 , & u^2 v^{-1} w^{-1} u^{-1} , & u v u y^{-1} u^{-2} , &  u v w^{-1} x^{-2} , & u v x v^{-1} u^{-2} \\
  u v y v^{-2} , & u w u z^{-1} u^{-2} , & u w x^{-1} y^{-2} , & u w z w^{-2} , & u x u v^{-1} u^{-2} , & u x y^{-1} z^{-2} \\
  u y u w^{-1} u^{-2} , &  u^{-2} v z u , & u^{-2} v^{-1} w^{-1} u^{-2} , & u^{-1} v^{-1} u^{-1} y^{-1} u^{-2} , & u^{-1} v^{-1} x^{-1} v^2 .  & 
\end{array}
$$

The gluing orbits of face classes from $\mathcal{D}_{\WS}$ are given by

\begin{center}
  \begin{tabular}{l|l}
        \toprule
	face & orbit \\
	\midrule
	$u$&$( 1, 2,14,20, 3)( 4,18,47,24, 7)( 5,12,17,46,23)( 6,19,48,25, 9)( 8,15,49,21,11)(10,16,50,22,13)$ \\
           &$(26,54,43,37,27)(28,42,52,36,29)(30,33,39,38,51)(31,44,53,40,32)(34,55,45,41,35)(56,59,60,58,57)$ \\
\midrule
	$v$&$( 1, 4,26,30, 5)( 2,15,51,32, 6)( 3,13,28,54,21)( 7,29,56,33,11)( 8,27,57,31,12)( 9,10,18,49,23)$ \\
           &$(14,46,40,35,16)(17,38,58,34,19)(20,25,41,42,47)(22,45,59,43,24)(36,55,44,39,37)(48,53,60,52,50)$ \\
\midrule
	$w$&$( 1, 6,34,36, 7)( 2,16,52,37, 8)( 3, 5,31,55,22)( 4,10,35,58,27)( 9,32,57,29,13)(11,12,19,50,24)$ \\
           &$(14,47,43,39,17)(15,18,42,60,38)(20,21,33,44,48)(23,30,56,45,25)(26,28,41,40,51)(46,49,54,59,53)$ \\
\midrule
	$x$&$( 1, 8,38,40, 9)( 2,17,53,41,10)( 3, 7,27,51,23)( 4,15,46,25,13)( 5,11,37,58,32)( 6,12,39,60,35)$ \\
           &$(14,48,45,28,18)(16,19,44,59,42)(20,22,29,26,49)(21,24,36,57,30)(31,33,43,52,34)(47,50,55,56,54)$ \\
\midrule
	$y$&$( 1,10,42,43,11)( 2,18,54,33,12)( 3, 9,35,52,24)( 4,28,59,39, 8)( 5, 6,16,47,21)( 7,13,41,60,37)$ \\
           &$(14,49,30,31,19)(15,26,56,44,17)(20,23,32,34,50)(22,25,40,58,36)(27,29,45,53,38)(46,51,57,55,48)$ \\
\midrule
	$z$&$( 1,12,44,45,13)( 2,19,55,29, 4)( 3,11,39,53,25)( 5,33,59,41, 9)( 6,31,56,28,10) ( 7, 8,17,48,22)$ \\
           &$(14,50,36,27,15)(16,34,57,26,18)(20,24,37,38,46)(21,43,60,40,23)(30,54,42,35,32)(47,52,58,51,49)$ \\
        \bottomrule
  \end{tabular}
\end{center}

\bigskip

\bigskip

The $60$ surfaces are given by their pentagonal disks. Every pentagonal disk is denoted by the corresponding pentagonal face in the lift of $\mathcal{D}_{\WS}$ it is parallel to, and the index of the dodecahedron it is contained in. The labelling follows Figure 2.

$$ \begin{array}{llllllll}
S_{1}&=\Big\langle&(u,\odot)_{60},&
(x,\odot)_{42},&
(w,\odot)_{39},&
(v,\odot)_{41},&
(z,\odot)_{43},&
(y,\odot)_{53},\\ 
&&(y,\otimes)_{54},&
(w,\otimes)_{28},&
(x,\otimes)_{44},&
(v,\otimes)_{33},&
(z,\otimes)_{45},&
(u,\otimes)_{56} 
 \quad\Big\rangle
\end{array} $$

$$ \begin{array}{llllllll}
S_{2}&=\Big\langle&(x,\otimes)_{60},&
(w,\odot)_{52},&
(u,\otimes)_{41},&
(y,\odot)_{58},&
(v,\otimes)_{42},&
(z,\otimes)_{40},\\ 
&&(z,\odot)_{16},&
(u,\odot)_{34},&
(w,\otimes)_{9},&
(y,\otimes)_{10},&
(v,\odot)_{32},&
(x,\odot)_{6} 
 \quad\Big\rangle
\end{array} $$

$$ \begin{array}{llllllll}
S_{3}&=\Big\langle&(y,\otimes)_{60},&
(x,\odot)_{40},&
(u,\otimes)_{39},&
(z,\odot)_{58},&
(w,\otimes)_{53},&
(v,\otimes)_{37},\\ 
&&(v,\odot)_{46},&
(u,\odot)_{51},&
(x,\otimes)_{8},&
(z,\otimes)_{17},&
(w,\odot)_{27},&
(y,\odot)_{15} 
 \quad\Big\rangle
\end{array} $$

$$ \begin{array}{llllllll}
S_{4}&=\Big\langle&(z,\otimes)_{60},&
(y,\odot)_{37},&
(u,\otimes)_{42},&
(v,\odot)_{58},&
(x,\otimes)_{43},&
(w,\otimes)_{35},\\ 
&&(w,\odot)_{24},&
(u,\odot)_{36},&
(v,\otimes)_{47},&
(y,\otimes)_{16},&
(x,\odot)_{34},&
(z,\odot)_{50} 
 \quad\Big\rangle
\end{array} $$

$$ \begin{array}{llllllll}
S_{5}&=\Big\langle&(v,\otimes)_{60},&
(z,\odot)_{35},&
(u,\otimes)_{53},&
(w,\odot)_{58},&
(y,\otimes)_{41},&
(x,\otimes)_{38},\\ 
&&(x,\odot)_{9},&
(u,\odot)_{32},&
(w,\otimes)_{25},&
(z,\otimes)_{46},&
(y,\odot)_{51},&
(v,\odot)_{23} 
 \quad\Big\rangle
\end{array} $$

$$ \begin{array}{llllllll}
S_{6}&=\Big\langle&(w,\otimes)_{60},&
(v,\odot)_{38},&
(u,\otimes)_{43},&
(x,\odot)_{58},&
(z,\otimes)_{39},&
(y,\otimes)_{52},\\ 
&&(y,\odot)_{8},&
(u,\odot)_{27},&
(v,\otimes)_{24},&
(x,\otimes)_{11},&
(z,\odot)_{36},&
(w,\odot)_{7} 
 \quad\Big\rangle
\end{array} $$

$$ \begin{array}{llllllll}
S_{7}&=\Big\langle&(v,\odot)_{60},&
(x,\odot)_{52},&
(y,\odot)_{39},&
(w,\otimes)_{42},&
(u,\odot)_{37},&
(z,\otimes)_{59},\\ 
&&(z,\odot)_{24},&
(y,\otimes)_{47},&
(w,\odot)_{11},&
(x,\otimes)_{33},&
(u,\otimes)_{54},&
(v,\otimes)_{21} 
 \quad\Big\rangle
\end{array} $$

$$ \begin{array}{llllllll}
S_{8}&=\Big\langle&(w,\odot)_{60},&
(y,\odot)_{40},&
(z,\odot)_{42},&
(x,\otimes)_{53},&
(u,\odot)_{35},&
(v,\otimes)_{59},\\ 
&&(v,\odot)_{9},&
(z,\otimes)_{25},&
(x,\odot)_{10},&
(y,\otimes)_{28},&
(u,\otimes)_{45},&
(w,\otimes)_{13} 
 \quad\Big\rangle
\end{array} $$

$$ \begin{array}{llllllll}
S_{9}&=\Big\langle&(x,\odot)_{60},&
(z,\odot)_{37},&
(v,\odot)_{53},&
(y,\otimes)_{43},&
(u,\odot)_{38},&
(w,\otimes)_{59},\\ 
&&(w,\odot)_{8},&
(v,\otimes)_{11},&
(y,\odot)_{17},&
(z,\otimes)_{44},&
(u,\otimes)_{33},&
(x,\otimes)_{12} 
 \quad\Big\rangle
\end{array} $$

$$ \begin{array}{llllllll}
S_{10}&=\Big\langle&(y,\odot)_{60},&
(v,\odot)_{35},&
(w,\odot)_{43},&
(z,\otimes)_{41},&
(u,\odot)_{52},&
(x,\otimes)_{59},\\ 
&&(x,\odot)_{16},&
(w,\otimes)_{10},&
(z,\odot)_{47},&
(v,\otimes)_{54},&
(u,\otimes)_{28},&
(y,\otimes)_{18} 
 \quad\Big\rangle
\end{array} $$

$$ \begin{array}{llllllll}
S_{11}&=\Big\langle&(z,\odot)_{60},&
(w,\odot)_{38},&
(x,\odot)_{41},&
(v,\otimes)_{39},&
(u,\odot)_{40},&
(y,\otimes)_{59},\\ 
&&(y,\odot)_{46},&
(x,\otimes)_{17},&
(v,\odot)_{25},&
(w,\otimes)_{45},&
(u,\otimes)_{44},&
(z,\otimes)_{48} 
 \quad\Big\rangle
\end{array} $$

$$ \begin{array}{llllllll}
S_{12}&=\Big\langle&(u,\otimes)_{60},&
(v,\otimes)_{52},&
(w,\otimes)_{40},&
(x,\otimes)_{37},&
(y,\otimes)_{35},&
(z,\otimes)_{38},\\ 
&&(z,\odot)_{34},&
(w,\odot)_{36},&
(v,\odot)_{51},&
(x,\odot)_{32},&
(y,\odot)_{27},&
(u,\odot)_{57} 
 \quad\Big\rangle
\end{array} $$

$$ \begin{array}{llllllll}
S_{13}&=\Big\langle&(u,\odot)_{59},&
(x,\odot)_{54},&
(w,\odot)_{44},&
(v,\odot)_{28},&
(z,\odot)_{33},&
(y,\odot)_{45},\\ 
&&(y,\otimes)_{30},&
(w,\otimes)_{26},&
(x,\otimes)_{55},&
(v,\otimes)_{31},&
(z,\otimes)_{29},&
(u,\otimes)_{57} 
 \quad\Big\rangle
\end{array} $$

$$ \begin{array}{llllllll}
S_{14}&=\Big\langle&(v,\odot)_{59},&
(x,\odot)_{43},&
(y,\odot)_{44},&
(w,\otimes)_{54},&
(u,\odot)_{39},&
(z,\otimes)_{56},\\ 
&&(z,\odot)_{11},&
(y,\otimes)_{21},&
(w,\odot)_{12},&
(x,\otimes)_{31},&
(u,\otimes)_{30},&
(v,\otimes)_{5} 
 \quad\Big\rangle
\end{array} $$

$$ \begin{array}{llllllll}
S_{15}&=\Big\langle&(w,\odot)_{59},&
(y,\odot)_{41},&
(z,\odot)_{54},&
(x,\otimes)_{45},&
(u,\odot)_{42},&
(v,\otimes)_{56},\\ 
&&(v,\odot)_{10},&
(z,\otimes)_{13},&
(x,\odot)_{18},&
(y,\otimes)_{26},&
(u,\otimes)_{29},&
(w,\otimes)_{4} 
 \quad\Big\rangle
\end{array} $$

$$ \begin{array}{llllllll}
S_{16}&=\Big\langle&(x,\odot)_{59},&
(z,\odot)_{39},&
(v,\odot)_{45},&
(y,\otimes)_{33},&
(u,\odot)_{53},&
(w,\otimes)_{56},\\ 
&&(w,\odot)_{17},&
(v,\otimes)_{12},&
(y,\odot)_{48},&
(z,\otimes)_{55},&
(u,\otimes)_{31},&
(x,\otimes)_{19} 
 \quad\Big\rangle
\end{array} $$

$$ \begin{array}{llllllll}
S_{17}&=\Big\langle&(y,\odot)_{59},&
(v,\odot)_{42},&
(w,\odot)_{33},&
(z,\otimes)_{28},&
(u,\odot)_{43},&
(x,\otimes)_{56},\\ 
&&(x,\odot)_{47},&
(w,\otimes)_{18},&
(z,\odot)_{21},&
(v,\otimes)_{30},&
(u,\otimes)_{26},&
(y,\otimes)_{49} 
 \quad\Big\rangle
\end{array} $$

$$ \begin{array}{llllllll}
S_{18}&=\Big\langle&(z,\odot)_{59},&
(w,\odot)_{53},&
(x,\odot)_{28},&
(v,\otimes)_{44},&
(u,\odot)_{41},&
(y,\otimes)_{56},\\ 
&&(y,\odot)_{25},&
(x,\otimes)_{48},&
(v,\odot)_{13},&
(w,\otimes)_{29},&
(u,\otimes)_{55},&
(z,\otimes)_{22} 
 \quad\Big\rangle
\end{array} $$

$$ \begin{array}{llllllll}
S_{19}&=\Big\langle&(u,\otimes)_{59},&
(v,\otimes)_{43},&
(w,\otimes)_{41},&
(x,\otimes)_{39},&
(y,\otimes)_{42},&
(z,\otimes)_{53},\\ 
&&(z,\odot)_{52},&
(w,\odot)_{37},&
(v,\odot)_{40},&
(x,\odot)_{35},&
(y,\odot)_{38},&
(u,\odot)_{58} 
 \quad\Big\rangle
\end{array} $$

$$ \begin{array}{llllllll}
S_{20}&=\Big\langle&(x,\otimes)_{58},&
(w,\odot)_{34},&
(u,\otimes)_{40},&
(y,\odot)_{57},&
(v,\otimes)_{35},&
(z,\otimes)_{51},\\ 
&&(z,\odot)_{6},&
(u,\odot)_{31},&
(w,\otimes)_{23},&
(y,\otimes)_{9},&
(v,\odot)_{30},&
(x,\odot)_{5} 
 \quad\Big\rangle
\end{array} $$

$$ \begin{array}{llllllll}
S_{21}&=\Big\langle&(y,\otimes)_{58},&
(x,\odot)_{51},&
(u,\otimes)_{37},&
(z,\odot)_{57},&
(w,\otimes)_{38},&
(v,\otimes)_{36},\\ 
&&(v,\odot)_{15},&
(u,\odot)_{26},&
(x,\otimes)_{7},&
(z,\otimes)_{8},&
(w,\odot)_{29},&
(y,\odot)_{4} 
 \quad\Big\rangle
\end{array} $$

$$ \begin{array}{llllllll}
S_{22}&=\Big\langle&(z,\otimes)_{58},&
(y,\odot)_{36},&
(u,\otimes)_{35},&
(v,\odot)_{57},&
(x,\otimes)_{52},&
(w,\otimes)_{32},\\ 
&&(w,\odot)_{50},&
(u,\odot)_{55},&
(v,\otimes)_{16},&
(y,\otimes)_{6},&
(x,\odot)_{31},&
(z,\odot)_{19} 
 \quad\Big\rangle
\end{array} $$

$$ \begin{array}{llllllll}
S_{23}&=\Big\langle&(v,\otimes)_{58},&
(z,\odot)_{32},&
(u,\otimes)_{38},&
(w,\odot)_{57},&
(y,\otimes)_{40},&
(x,\otimes)_{27},\\ 
&&(x,\odot)_{23},&
(u,\odot)_{30},&
(w,\otimes)_{46},&
(z,\otimes)_{15},&
(y,\odot)_{26},&
(v,\odot)_{49} 
 \quad\Big\rangle
\end{array} $$

$$ \begin{array}{llllllll}
S_{24}&=\Big\langle&(w,\otimes)_{58},&
(v,\odot)_{27},&
(u,\otimes)_{52},&
(x,\odot)_{57},&
(z,\otimes)_{37},&
(y,\otimes)_{34},\\ 
&&(y,\odot)_{7},&
(u,\odot)_{29},&
(v,\otimes)_{50},&
(x,\otimes)_{24},&
(z,\odot)_{55},&
(w,\odot)_{22} 
 \quad\Big\rangle
\end{array} $$

$$ \begin{array}{llllllll}
S_{25}&=\Big\langle&(u,\otimes)_{58},&
(v,\otimes)_{34},&
(w,\otimes)_{51},&
(x,\otimes)_{36},&
(y,\otimes)_{32},&
(z,\otimes)_{27},\\ 
&&(z,\odot)_{31},&
(w,\odot)_{55},&
(v,\odot)_{26},&
(x,\odot)_{30},&
(y,\odot)_{29},&
(u,\odot)_{56} 
 \quad\Big\rangle
\end{array} $$

$$ \begin{array}{llllllll}
S_{26}&=\Big\langle&(x,\otimes)_{57},&
(w,\odot)_{31},&
(u,\otimes)_{51},&
(y,\odot)_{56},&
(v,\otimes)_{32},&
(z,\otimes)_{26},\\ 
&&(z,\odot)_{5},&
(u,\odot)_{33},&
(w,\otimes)_{49},&
(y,\otimes)_{23},&
(v,\odot)_{54},&
(x,\odot)_{21} 
 \quad\Big\rangle
\end{array} $$

$$ \begin{array}{llllllll}
S_{27}&=\Big\langle&(y,\otimes)_{57},&
(x,\odot)_{26},&
(u,\otimes)_{36},&
(z,\odot)_{56},&
(w,\otimes)_{27},&
(v,\otimes)_{55},\\ 
&&(v,\odot)_{4},&
(u,\odot)_{28},&
(x,\otimes)_{22},&
(z,\otimes)_{7},&
(w,\odot)_{45},&
(y,\odot)_{13} 
 \quad\Big\rangle
\end{array} $$

$$ \begin{array}{llllllll}
S_{28}&=\Big\langle&(z,\otimes)_{57},&
(y,\odot)_{55},&
(u,\otimes)_{32},&
(v,\odot)_{56},&
(x,\otimes)_{34},&
(w,\otimes)_{30},\\ 
&&(w,\odot)_{19},&
(u,\odot)_{44},&
(v,\otimes)_{6},&
(y,\otimes)_{5},&
(x,\odot)_{33},&
(z,\odot)_{12} 
 \quad\Big\rangle
\end{array} $$

$$ \begin{array}{llllllll}
S_{29}&=\Big\langle&(v,\otimes)_{57},&
(z,\odot)_{30},&
(u,\otimes)_{27},&
(w,\odot)_{56},&
(y,\otimes)_{51},&
(x,\otimes)_{29},\\ 
&&(x,\odot)_{49},&
(u,\odot)_{54},&
(w,\otimes)_{15},&
(z,\otimes)_{4},&
(y,\odot)_{28},&
(v,\odot)_{18} 
 \quad\Big\rangle
\end{array} $$

$$ \begin{array}{llllllll}
S_{30}&=\Big\langle&(w,\otimes)_{57},&
(v,\odot)_{29},&
(u,\otimes)_{34},&
(x,\odot)_{56},&
(z,\otimes)_{36},&
(y,\otimes)_{31},\\ 
&&(y,\odot)_{22},&
(u,\odot)_{45},&
(v,\otimes)_{19},&
(x,\otimes)_{50},&
(z,\odot)_{44},&
(w,\odot)_{48} 
 \quad\Big\rangle
\end{array} $$

$$ \begin{array}{llllllll}
S_{31}&=\Big\langle&(y,\otimes)_{55},&
(x,\odot)_{29},&
(u,\otimes)_{50},&
(z,\odot)_{45},&
(w,\otimes)_{36},&
(v,\otimes)_{48},\\ 
&&(v,\odot)_{7},&
(u,\odot)_{13},&
(x,\otimes)_{20},&
(z,\otimes)_{24},&
(w,\odot)_{25},&
(y,\odot)_{3} 
 \quad\Big\rangle
\end{array} $$

$$ \begin{array}{llllllll}
S_{32}&=\Big\langle&(w,\otimes)_{55},&
(v,\odot)_{22},&
(u,\otimes)_{19},&
(x,\odot)_{45},&
(z,\otimes)_{50},&
(y,\otimes)_{44},\\ 
&&(y,\odot)_{20},&
(u,\odot)_{25},&
(v,\otimes)_{17},&
(x,\otimes)_{14},&
(z,\odot)_{53},&
(w,\odot)_{46} 
 \quad\Big\rangle
\end{array} $$

$$ \begin{array}{llllllll}
S_{33}&=\Big\langle&(v,\odot)_{55},&
(x,\odot)_{44},&
(y,\odot)_{50},&
(w,\otimes)_{31},&
(u,\odot)_{48},&
(z,\otimes)_{34},\\ 
&&(z,\odot)_{17},&
(y,\otimes)_{12},&
(w,\odot)_{14},&
(x,\otimes)_{16},&
(u,\otimes)_{6},&
(v,\otimes)_{2} 
 \quad\Big\rangle
\end{array} $$

$$ \begin{array}{llllllll}
S_{34}&=\Big\langle&(x,\odot)_{55},&
(z,\odot)_{48},&
(v,\odot)_{36},&
(y,\otimes)_{19},&
(u,\odot)_{22},&
(w,\otimes)_{34},\\ 
&&(w,\odot)_{20},&
(v,\otimes)_{14},&
(y,\odot)_{24},&
(z,\otimes)_{52},&
(u,\otimes)_{16},&
(x,\otimes)_{47} 
 \quad\Big\rangle
\end{array} $$

$$ \begin{array}{llllllll}
S_{35}&=\Big\langle&(x,\otimes)_{54},&
(w,\odot)_{21},&
(u,\otimes)_{18},&
(y,\odot)_{43},&
(v,\otimes)_{49},&
(z,\otimes)_{42},\\ 
&&(z,\odot)_{20},&
(u,\odot)_{24},&
(w,\otimes)_{16},&
(y,\otimes)_{14},&
(v,\odot)_{52},&
(x,\odot)_{50} 
 \quad\Big\rangle
\end{array} $$

$$ \begin{array}{llllllll}
S_{36}&=\Big\langle&(z,\otimes)_{54},&
(y,\odot)_{33},&
(u,\otimes)_{49},&
(v,\odot)_{43},&
(x,\otimes)_{30},&
(w,\otimes)_{47},\\ 
&&(w,\odot)_{5},&
(u,\odot)_{11},&
(v,\otimes)_{23},&
(y,\otimes)_{20},&
(x,\odot)_{24},&
(z,\odot)_{3} 
 \quad\Big\rangle
\end{array} $$

$$ \begin{array}{llllllll}
S_{37}&=\Big\langle&(w,\odot)_{54},&
(y,\odot)_{42},&
(z,\odot)_{49},&
(x,\otimes)_{28},&
(u,\odot)_{47},&
(v,\otimes)_{26},\\ 
&&(v,\odot)_{16},&
(z,\otimes)_{10},&
(x,\odot)_{14},&
(y,\otimes)_{15},&
(u,\otimes)_{4},&
(w,\otimes)_{2} 
 \quad\Big\rangle
\end{array} $$

$$ \begin{array}{llllllll}
S_{38}&=\Big\langle&(y,\odot)_{54},&
(v,\odot)_{47},&
(w,\odot)_{30},&
(z,\otimes)_{18},&
(u,\odot)_{21},&
(x,\otimes)_{26},\\ 
&&(x,\odot)_{20},&
(w,\otimes)_{14},&
(z,\odot)_{23},&
(v,\otimes)_{51},&
(u,\otimes)_{15},&
(y,\otimes)_{46} 
 \quad\Big\rangle
\end{array} $$

$$ \begin{array}{llllllll}
S_{39}&=\Big\langle&(y,\otimes)_{53},&
(x,\odot)_{25},&
(u,\otimes)_{17},&
(z,\odot)_{40},&
(w,\otimes)_{48},&
(v,\otimes)_{38},\\ 
&&(v,\odot)_{20},&
(u,\odot)_{23},&
(x,\otimes)_{15},&
(z,\otimes)_{14},&
(w,\odot)_{51},&
(y,\odot)_{49} 
 \quad\Big\rangle
\end{array} $$

$$ \begin{array}{llllllll}
S_{40}&=\Big\langle&(v,\otimes)_{53},&
(z,\odot)_{41},&
(u,\otimes)_{48},&
(w,\odot)_{40},&
(y,\otimes)_{45},&
(x,\otimes)_{46},\\ 
&&(x,\odot)_{13},&
(u,\odot)_{9},&
(w,\otimes)_{22},&
(z,\otimes)_{20},&
(y,\odot)_{23},&
(v,\odot)_{3} 
 \quad\Big\rangle
\end{array} $$

$$ \begin{array}{llllllll}
S_{41}&=\Big\langle&(x,\odot)_{53},&
(z,\odot)_{38},&
(v,\odot)_{48},&
(y,\otimes)_{39},&
(u,\odot)_{46},&
(w,\otimes)_{44},\\ 
&&(w,\odot)_{15},&
(v,\otimes)_{8},&
(y,\odot)_{14},&
(z,\otimes)_{19},&
(u,\otimes)_{12},&
(x,\otimes)_{2} 
 \quad\Big\rangle
\end{array} $$

$$ \begin{array}{llllllll}
S_{42}&=\Big\langle&(w,\otimes)_{52},&
(v,\odot)_{37},&
(u,\otimes)_{47},&
(x,\odot)_{36},&
(z,\otimes)_{43},&
(y,\otimes)_{50},\\ 
&&(y,\odot)_{11},&
(u,\odot)_{7},&
(v,\otimes)_{20},&
(x,\otimes)_{21},&
(z,\odot)_{22},&
(w,\odot)_{3} 
 \quad\Big\rangle
\end{array} $$

$$ \begin{array}{llllllll}
S_{43}&=\Big\langle&(y,\odot)_{52},&
(v,\odot)_{34},&
(w,\odot)_{47},&
(z,\otimes)_{35},&
(u,\odot)_{50},&
(x,\otimes)_{42},\\ 
&&(x,\odot)_{19},&
(w,\otimes)_{6},&
(z,\odot)_{14},&
(v,\otimes)_{18},&
(u,\otimes)_{10},&
(y,\otimes)_{2} 
 \quad\Big\rangle
\end{array} $$

$$ \begin{array}{llllllll}
S_{44}&=\Big\langle&(x,\otimes)_{51},&
(w,\odot)_{32},&
(u,\otimes)_{46},&
(y,\odot)_{30},&
(v,\otimes)_{40},&
(z,\otimes)_{49},\\ 
&&(z,\odot)_{9},&
(u,\odot)_{5},&
(w,\otimes)_{20},&
(y,\otimes)_{25},&
(v,\odot)_{21},&
(x,\odot)_{3} 
 \quad\Big\rangle
\end{array} $$

$$ \begin{array}{llllllll}
S_{45}&=\Big\langle&(z,\odot)_{51},&
(w,\odot)_{26},&
(x,\odot)_{46},&
(v,\otimes)_{27},&
(u,\odot)_{49},&
(y,\otimes)_{38},\\ 
&&(y,\odot)_{18},&
(x,\otimes)_{4},&
(v,\odot)_{14},&
(w,\otimes)_{17},&
(u,\otimes)_{8},&
(z,\otimes)_{2} 
 \quad\Big\rangle
\end{array} $$

$$ \begin{array}{llllllll}
S_{46}&=\Big\langle&(w,\otimes)_{50},&
(v,\odot)_{24},&
(u,\otimes)_{14},&
(x,\odot)_{22},&
(z,\otimes)_{47},&
(y,\otimes)_{48},\\ 
&&(y,\odot)_{21},&
(u,\odot)_{3},&
(v,\otimes)_{46},&
(x,\otimes)_{49},&
(z,\odot)_{25},&
(w,\odot)_{23} 
 \quad\Big\rangle
\end{array} $$

$$ \begin{array}{llllllll}
S_{47}&=\Big\langle&(v,\odot)_{50},&
(x,\odot)_{48},&
(y,\odot)_{47},&
(w,\otimes)_{19},&
(u,\odot)_{20},&
(z,\otimes)_{16},\\ 
&&(z,\odot)_{46},&
(y,\otimes)_{17},&
(w,\odot)_{49},&
(x,\otimes)_{18},&
(u,\otimes)_{2},&
(v,\otimes)_{15} 
 \quad\Big\rangle
\end{array} $$

$$ \begin{array}{llllllll}
S_{48}&=\Big\langle&(v,\otimes)_{45},&
(z,\odot)_{28},&
(u,\otimes)_{22},&
(w,\odot)_{41},&
(y,\otimes)_{29},&
(x,\otimes)_{25},\\ 
&&(x,\odot)_{4},&
(u,\odot)_{10},&
(w,\otimes)_{7},&
(z,\otimes)_{3},&
(y,\odot)_{9},&
(v,\odot)_{1} 
 \quad\Big\rangle
\end{array} $$

$$ \begin{array}{llllllll}
S_{49}&=\Big\langle&(v,\odot)_{44},&
(x,\odot)_{39},&
(y,\odot)_{19},&
(w,\otimes)_{33},&
(u,\odot)_{17},&
(z,\otimes)_{31},\\ 
&&(z,\odot)_{8},&
(y,\otimes)_{11},&
(w,\odot)_{2},&
(x,\otimes)_{6},&
(u,\otimes)_{5},&
(v,\otimes)_{1} 
 \quad\Big\rangle
\end{array} $$

$$ \begin{array}{llllllll}
S_{50}&=\Big\langle&(w,\otimes)_{43},&
(v,\odot)_{39},&
(u,\otimes)_{21},&
(x,\odot)_{37},&
(z,\otimes)_{33},&
(y,\otimes)_{24},\\ 
&&(y,\odot)_{12},&
(u,\odot)_{8},&
(v,\otimes)_{3},&
(x,\otimes)_{5},&
(z,\odot)_{7},&
(w,\odot)_{1} 
 \quad\Big\rangle
\end{array} $$

$$ \begin{array}{llllllll}
S_{51}&=\Big\langle&(w,\odot)_{42},&
(y,\odot)_{35},&
(z,\odot)_{18},&
(x,\otimes)_{41},&
(u,\odot)_{16},&
(v,\otimes)_{28},\\ 
&&(v,\odot)_{6},&
(z,\otimes)_{9},&
(x,\odot)_{2},&
(y,\otimes)_{4},&
(u,\otimes)_{13},&
(w,\otimes)_{1} 
 \quad\Big\rangle
\end{array} $$

$$ \begin{array}{llllllll}
S_{52}&=\Big\langle&(v,\otimes)_{41},&
(z,\odot)_{10},&
(u,\otimes)_{25},&
(w,\odot)_{35},&
(y,\otimes)_{13},&
(x,\otimes)_{40},\\ 
&&(x,\odot)_{1},&
(u,\odot)_{6},&
(w,\otimes)_{3},&
(z,\otimes)_{23},&
(y,\odot)_{32},&
(v,\odot)_{5} 
 \quad\Big\rangle
\end{array} $$

$$ \begin{array}{llllllll}
S_{53}&=\Big\langle&(w,\otimes)_{39},&
(v,\odot)_{17},&
(u,\otimes)_{11},&
(x,\odot)_{38},&
(z,\otimes)_{12},&
(y,\otimes)_{37},\\ 
&&(y,\odot)_{2},&
(u,\odot)_{15},&
(v,\otimes)_{7},&
(x,\otimes)_{1},&
(z,\odot)_{27},&
(w,\odot)_{4} 
 \quad\Big\rangle
\end{array} $$

$$ \begin{array}{llllllll}
S_{54}&=\Big\langle&(w,\otimes)_{37},&
(v,\odot)_{8},&
(u,\otimes)_{24},&
(x,\odot)_{27},&
(z,\otimes)_{11},&
(y,\otimes)_{36},\\ 
&&(y,\odot)_{1},&
(u,\odot)_{4},&
(v,\otimes)_{22},&
(x,\otimes)_{3},&
(z,\odot)_{29},&
(w,\odot)_{13} 
 \quad\Big\rangle
\end{array} $$

$$ \begin{array}{llllllll}
S_{55}&=\Big\langle&(x,\otimes)_{35},&
(w,\odot)_{16},&
(u,\otimes)_{9},&
(y,\odot)_{34},&
(v,\otimes)_{10},&
(z,\otimes)_{32},\\ 
&&(z,\odot)_{2},&
(u,\odot)_{19},&
(w,\otimes)_{5},&
(y,\otimes)_{1},&
(v,\odot)_{31},&
(x,\odot)_{12} 
 \quad\Big\rangle
\end{array} $$

$$ \begin{array}{llllllll}
S_{56}&=\Big\langle&(v,\odot)_{33},&
(x,\odot)_{11},&
(y,\odot)_{31},&
(w,\otimes)_{21},&
(u,\odot)_{12},&
(z,\otimes)_{30},\\ 
&&(z,\odot)_{1},&
(y,\otimes)_{3},&
(w,\odot)_{6},&
(x,\otimes)_{32},&
(u,\otimes)_{23},&
(v,\otimes)_{9} 
 \quad\Big\rangle
\end{array} $$

$$ \begin{array}{llllllll}
S_{57}&=\Big\langle&(v,\otimes)_{29},&
(z,\odot)_{26},&
(u,\otimes)_{7},&
(w,\odot)_{28},&
(y,\otimes)_{27},&
(x,\otimes)_{13},\\ 
&&(x,\odot)_{15},&
(u,\odot)_{18},&
(w,\otimes)_{8},&
(z,\otimes)_{1},&
(y,\odot)_{10},&
(v,\odot)_{2} 
 \quad\Big\rangle
\end{array} $$

$$ \begin{array}{llllllll}
S_{58}&=\Big\langle&(v,\otimes)_{25},&
(z,\odot)_{13},&
(u,\otimes)_{20},&
(w,\odot)_{9},&
(y,\otimes)_{22},&
(x,\otimes)_{23},\\ 
&&(x,\odot)_{7},&
(u,\odot)_{1},&
(w,\otimes)_{24},&
(z,\otimes)_{21},&
(y,\odot)_{5},&
(v,\odot)_{11} 
 \quad\Big\rangle
\end{array} $$

$$ \begin{array}{llllllll}
S_{59}&=\Big\langle&(v,\odot)_{19},&
(x,\odot)_{17},&
(y,\odot)_{16},&
(w,\otimes)_{12},&
(u,\odot)_{14},&
(z,\otimes)_{6},\\ 
&&(z,\odot)_{15},&
(y,\otimes)_{8},&
(w,\odot)_{18},&
(x,\otimes)_{10},&
(u,\otimes)_{1},&
(v,\otimes)_{4} 
 \quad\Big\rangle
\end{array} $$

$$ \begin{array}{llllllll}
S_{60}&=\Big\langle&(v,\otimes)_{13},&
(z,\odot)_{4},&
(u,\otimes)_{3},&
(w,\odot)_{10},&
(y,\otimes)_{7},&
(x,\otimes)_{9},\\ 
&&(x,\odot)_{8},&
(u,\odot)_{2},&
(w,\otimes)_{11},&
(z,\otimes)_{5},&
(y,\odot)_{6},&
(v,\odot)_{12} 
 \quad\Big\rangle
\end{array} $$
                           
\end{document}